\pdfoutput=1
\documentclass{amsart}
\usepackage{amsmath}
\usepackage{amsfonts}
\usepackage{amssymb}
\usepackage{amsthm}
\usepackage{bm}
\usepackage{mathrsfs}
\usepackage{stmaryrd}
\usepackage{tikz,tikz-cd}
\usetikzlibrary{arrows,chains,matrix,positioning,scopes}
\usepackage[numbers]{natbib}
\usepackage{slashed}
\usepackage{todonotes}
\usepackage{graphicx}
\usepackage{microtype}
\usepackage{bbm}
\usepackage{pdflscape}
\usepackage[pagebackref=true]{hyperref}

\renewcommand*{\backrefalt}[4]{%
	\ifcase #1 \footnotesize{(Not cited.)}%
	\or	\footnotesize{(Cited on page~#2)}
	\else	\footnotesize{(Cited on pages~#2)}%
	\fi}

\usepackage{etex}
\usepackage{etoolbox}
\patchcmd{\thebibliography}{*}{}{}{}

\newtheorem{theorem}{Theorem}[section]
\newtheorem{lemma}[theorem]{Lemma}
\newtheorem{proposition}[theorem]{Proposition}
\newtheorem{corollary}[theorem]{Corollary}
{\theoremstyle{definition} \newtheorem{definition}{Definition}[section]}
{\theoremstyle{definition} \newtheorem{construction}{Construction}[section]}
{\theoremstyle{definition} }
{\theoremstyle{definition} \newtheorem*{exercise*}{Exercise}}
{\theoremstyle{definition}\newtheorem{example}{Example}}
{\theoremstyle{definition} }
{\theoremstyle{remark}\newtheorem*{remark}{Remark}}
{\theoremstyle{remark}}

{\theoremstyle{remark}}
{\theoremstyle{definition}}
{\theoremstyle{definition}}

\newcommand{\F}{\mathbb{F}}
\newcommand{\R}{\mathbb{R}}

\newcommand{\id}{\mathbb{I}}

\newcommand{\Ext}{\mathrm{Ext}}

\newcommand{\Mor}{\mathrm{Mor}}
\newcommand{\CF}{\mathit{CF}}

\newcommand{\CFD}{\mathit{CFD}}
\newcommand{\CFAA}{\mathit{CFAA}}
\newcommand{\CFDA}{\mathit{CFDA}}

\newcommand{\HF}{\mathit{HF}}

\newcommand{\AZ}{\mathsf{AZ}}
\newcommand{\AT}{\mathsf{AT}}

\newcommand{\ind}{\mathrm{ind}}

\title{Composition maps in Heegaard Floer homology}
\author{Jesse Cohen}
\address{Department of Mathematics, University of Oregon, Eugene, OR 97403-1222, USA}
\email{jcohen2@uoregon.edu}
\thanks{This material is based upon work supported by the National Science Foundation under Grant No. DMS-2204214 and under Grant No. DMS-1928930, while the author was in residence at the Simons Laufer Mathematical Sciences Institute (previously known as MSRI) in Berkeley, California, during the Fall 2022 semester.}
\begin{document}
	\begin{abstract}
		We use results of Auroux \cite{Auroux2010} and Zemke \cite{Zemke2021} to prove that, in the morphism spaces formulation of Heegaard Floer homology given in \cite{Lipshitz_2011}, the opposite composition map agrees up to homotopy with the map on Heegaard Floer complexes induced by a pair-of-pants cobordism. As an application, we give an algorithm for computing arbitrary cobordism maps on hat Heegaard Floer homology.
	\end{abstract}
	\maketitle
	\section{Introduction}
	Heegaard Floer homology is a suite of invariants of closed oriented 3-manifolds and cobordisms between them introduced by Peter Ozsv\'{a}th and Zolt\'{a}n Szab\'{o} in \cite{OzsSz2004} (see also \cite{Ozsvath2006}). The particular variant of Heegaard Floer homology we will be concerned with is the so-called `hat' version. This invariant associates to a closed oriented 3-manifold $Y$ a graded $\F_2$-vector space $\widehat{\HF}(Y)$ and to each smooth, connected, 4-dimensional cobordism $W:Y_0\to Y_1$ a map $\widehat{F}_W:\widehat{\HF}(Y_0)\to\widehat{\HF}(Y_1)$ and this assignment is functorial with respect to composition of cobordisms. The vector space $\widehat{\HF}(Y)$ is the homology of a complex $\widehat{\CF}(Y)$ defined as a variant of the Lagrangian-intersection Floer complex of a pair of Lagrangian tori in a K\"{a}hler manifold. Bordered Heegaard Floer homology, defined by Lipshitz--Ozsv\'{a}th--Thurston in \cite{Lipshitz2018}, is a suite of invariants associated to a 3-manifold $Y$ with parametrized boundary taking the form of homotopy types of $A_\infty$-modules over algebras $\mathcal{A}(\mathcal{Z})$ associated to a combinatorialization $\mathcal{Z}$ of the boundary parametrization. In particular, if $Y$ has one boundary component, the bordered Floer package gives us a left type-$D$ structure $\widehat{\CFD}(Y)$, which one may think of as a projective left dg-module, whose homotopy type is a smooth invariant of $Y$. We briefly recall the construction of this object in Section 2. These modules satisfy pairing theorems as follows: if $Y_1$ and $Y_2$ are 3-manifolds with the same connected boundary surface and $Y_{12}=-Y_1\cup_\partial Y_2$ is the closed 3-manifold obtained by gluing $Y_1$ and $Y_2$ along their respective boundary parametrizations, then there is a homotopy equivalence  $\widehat{\CF}(Y_{12})\simeq\Mor^{\mathcal{A}}(\widehat{\CFD}(Y_1),\widehat{\CFD}(Y_2))$, where the right-hand side is the space of $\mathcal{A}=\mathcal{A}(-\mathcal{Z})$-module homomorphisms $\widehat{\CFD}(Y_1)\to\widehat{\CFD}(Y_2)$.
	\subsection{Results}
	Our main result is the following chain-level version of an assertion of Lipshitz--Oszv\'{a}th--Thurston given in \cite[Section 1.5]{Lipshitz_2011}.
	\begin{theorem}\label{MainTheorem}
		Let $Y_1$, $Y_2$, and $Y_3$ be bordered 3-manifolds, all of which have boundaries parametrized by the same surface $F$, and let $\mathcal{A}=\mathcal{A}(-F)$ be the algebra associated to $-F$. Let $Y_{ij}=-Y_i\cup_\partial Y_j$ and consider the pair of pants cobordism $W:Y_{12}\sqcup Y_{23}\to Y_{13}$ given by
		\begin{align}
			W=(\triangle\times F)\cup_{e_1\times F}(e_1\times Y_1)\cup_{e_2\times F}(e_2\times Y_2)\cup_{e_3\times F}(e_3\times Y_3),
		\end{align}
		where $\triangle$ is a triangle with edges $e_1$, $e_2$, and $e_3$ in cyclic order. If we define $\Mor^{\mathcal{A}}(Y_i,Y_j):=\Mor^{\mathcal{A}}(\widehat{\CFD}(Y_i),\widehat{\CFD}(Y_j))$ to be the space of left $\mathcal{A}$-module homomorphisms $\widehat{\CFD}(Y_i)\to\widehat{\CFD}(Y_j)$, then the composition map $f\otimes g\mapsto g\circ f$ fits into a homotopy commutative square of the form
		\begin{align}
			\begin{tikzcd}[ampersand replacement=\&,column sep=1.5cm]
				\Mor^{\mathcal{A}}(Y_1,Y_2)\otimes\Mor^{\mathcal{A}}(Y_2,Y_3)\arrow[d,"\simeq"']\arrow[r,"f\otimes g\mapsto g\circ f"] \& \Mor^{\mathcal{A}}(Y_1,Y_3)\arrow[d,"\simeq"]\\
				\widehat{\CF}(Y_{12})\otimes\widehat{\CF}(Y_{23})\arrow[r,"\widehat{f}_{W}"]\&\widehat{\CF}(Y_{13})
			\end{tikzcd}
		\end{align}
		where $\widehat{f}_{W}$ is the map induced by $W$ and the vertical maps come from the pairing theorem \cite[Theorem 1]{Lipshitz_2011}.
	\end{theorem}
	\begin{remark}
		The complexes $\Mor^{\mathcal{A}}(Y_i,Y_j)$ are well-defined only up to a choice of bordered Heegaard diagrams for $Y_i$ and $Y_j$.
	\end{remark}
	To prove this, we use a bordered Heegaard triple $\AT$, originally defined by Auroux in \cite{Auroux2010}. In particular, we prove the following.
	\begin{theorem}
		Let $\mathcal{H}_i$ be bordered Heegaard diagrams for bordered 3-manifolds $Y_i$ for $i=1,2,3$ and let $\mathcal{H}_i^+$ be the bordered Heegaard triple obtained by doubling the $\beta$-circles in $\mathcal{H}_i$ by a small Hamiltonian isotopy. Then the map
		\begin{align}
			\widehat{G}_{\AT}:\Mor^{\mathcal{A}}(Y_1,Y_2)\otimes\Mor^{\mathcal{A}}(Y_2,Y_3)\to\Mor^{\mathcal{A}}(Y_1,Y_3)
		\end{align}
		induced by counting pseudoholomorphic triangles in $\AT\cup\mathcal{H}_1^+\cup\mathcal{H}_2^+\cup\mathcal{H}_3^+$, identifying $\Mor^{\mathcal{A}}(Y_i,Y_j)$ with $\overline{\widehat{\CFD}(Y_i)}\boxtimes\mathcal{A}\boxtimes\widehat{\CFD}(Y_j)$, agrees up to homotopy with the composition map $f\otimes g\mapsto g\circ f$.
	\end{theorem}
	We then elaborate on the construction of 4-manifolds with boundary and corners from bordered Heegaard triples and show (Corollary \ref{PantsCorollary}) that the triple $\AT_{1,2,3}$ represents a variant of the pair of pants cobordism described above and use this to prove Theorem \ref{MainTheorem} via results of Zemke \cite{Zemke2021,Zemke2021b}. Lastly, as a consequence of Theorem \ref{MainTheorem}, we give a new algorithm for computing the map $\widehat{\HF}(Y_0)\to\widehat{\HF}(Y_1)$ associated to a cobordism $X:Y_0\to Y_1$, at the chain level, via composition of morphisms. This algorithm gives an alternative to the combinatorial approaches of \cite{LipshitzManolescuWang} and \cite{ManolescuOzsvathThurston2020grid}.
	\subsection*{Acknowledgments}
	The author would like to thank Gary Guth, Robert Lipshitz, Maggie Miller, and Dylan Thurston for their helpful comments and suggestions.
	\section{Background}
	\begin{definition}
		Fix a dg-algebra $\mathcal{A}$ over a ring $\Bbbk$. A (left) type-$D$ structure over $\mathcal{A}$ is a pair $(N,\delta^1_N)$ consisting of a graded $\Bbbk$-module $N$ and a map
		\begin{align}
			\delta_N^1:N\to(\mathcal{A}\otimes N)[1]
		\end{align}
		satisfying the compatibility condition
		\begin{align}
			(\mu_2\otimes\id_N)\circ(\id_\mathcal{A}\otimes\delta_N^1)\circ\delta_N^1+(\mu_1\otimes\id_N)\circ\delta_N^1=0.
		\end{align}
	\end{definition}
	A type-$D$ structure homomorphism is a $\Bbbk$-module map $f:N_1\to\mathcal{A}\otimes N_2$ satisfying the equation
	\begin{align}
		(\mu_2\otimes\id_{N_2})\circ(\id_\mathcal{A}\otimes f)\circ\delta_{N_1}^1+(\mu_2\otimes\id_{N_2})\circ(\id_\mathcal{A}\otimes\delta_{N_2}^1)\circ f+(\mu_1\otimes\id_{N_2})\circ f=0
	\end{align}
	and a homotopy between type-$D$ structure homomorphisms $f,g:N_1\to\mathcal{A}\otimes N_2$ is a $\Bbbk$-module homomorphism $h:N_1\to(\mathcal{A}\otimes N_2)[-1]$ such that
	\begin{align*}
		(\mu_2\otimes\id_{N_2})\circ(\id_{\mathcal{A}}\otimes h)\circ\delta_{N_1}^1+(\mu_2\otimes\id_{N_2})\circ(\id_{\mathcal{A}}\otimes\delta_{N_2}^1)\circ h+(\mu_1\otimes\id_{N_2})\circ h=f-g.
	\end{align*}
	\begin{example}
		Let $\mathcal{A}$ be a dg-algebra over $\Bbbk=\F_2^k$ and $M$ a dg-module over $\mathcal{A}$ which is free as an $\mathcal{A}$-module. Fix an $\mathcal{A}$-basis $B$ for $M$ and let $N=\Bbbk B$. Then the restriction of the differential $\partial_M$ to $N$ determines a map $\delta^1:N\to(\mathcal{A}\otimes N)[1]$ satisfying the type-$D$ structure relation. Any dg-module homomorphism $M_1\to M_2$ induces a corresponding map of type-$D$ structures and the converse is true for type-$D$ structures obtained in this manner. Similarly, homotopies of such maps are equivalent to homotopies of type-$D$ structures.
	\end{example}
	On the other hand, if $(N,\delta^1)$ is a left type-$D$ structure over $\mathcal{A}$, then $\mathcal{A}\otimes N$ is a left differential $\mathcal{A}$-module with differential
	\begin{align}
		m_1=(\mu_2\otimes\id_N)\circ(\id_{\mathcal{A}}\otimes\delta^1)+\mu_1\otimes\id_N
	\end{align}
	and module structure map $m_2=\mu_2\otimes\id_N$. As in the above example, type-$D$ homomorphisms and homotopies induce chain homomorphisms and chain homotopies, respectively.
	\begin{definition}
		Given a type-$D$ structure $(N,\delta^1)$, recursively define maps
		\begin{align}
			\delta^i:N\to(\mathcal{A}^{\otimes i}\otimes N)[i]
		\end{align}
		by taking $\delta^0=\id_N$ and $\delta^i=(\id_{\mathcal{A}^{\otimes i-1}}\otimes\delta^1)\circ\delta^{i-1}$. We say that $(N,\delta^1)$ is \emph{bounded} if, for all $x\in N$, there is an integer $n$ such that $\delta^i(x)=0$ for all $i\geq n$.
	\end{definition}
	\begin{definition}
		A \emph{pointed matched circle} $\mathcal{Z}$ is a quadruple $(Z,\bm{a},M,z)$ consisting of an oriented circle $Z$, a subset $\bm{a}\subset Z$ consisting of $4k$ points $a_1,\dots,a_{4k}$, a 2-to-1 function $M:\bm{a}\to[2k]$ called a \emph{matching}, and a basepoint $z\in Z\smallsetminus\bm{a}$. We require that surgering $Z$ along the matching $M$ yields a single circle; this ensures that the following construction is well-defined. Given a pointed matched circle $\mathcal{Z}$, there is an oriented surface $F(\mathcal{Z})$ obtained by attaching oriented 2-dimensional 1-handles to the points $\bm{a}\subset\partial D^2\subset D^2$ according to the matching $M$ to obtain a surface with a single boundary component, along which we glue a second disk.
	\end{definition}
	We will regard each pointed matched circle $\mathcal{Z}=(Z,\bm{a},M,z)$ as a contact 1-manifold and refer to intervals $\rho\subset Z$ which have ends on $\bm{a}$ and do not cross $z$ as \emph{Reeb chords}.
	\begin{definition}
		A \emph{bordered} $3$-\emph{manifold} $Y$ is an oriented $3$-manifold with boundary together with an orientation-preserving diffeomorphism $\phi:F(\mathcal{Z})\to\partial Y$. Such data can be specified by a \emph{bordered Heegaard diagram} $(\overline{\Sigma},\bm{\alpha},\bm{\beta},z)$, where:
		\begin{itemize}
			\item $\overline{\Sigma}$ is a compact, oriented, surface of some genus $g$,
			\item $\bm{\alpha}=\bm{\alpha}^a\cup\bm{\alpha}^c=\{\alpha_1^a,\dots,\alpha_{2k}^a,\alpha_1^c,\dots,\alpha_{g-k}^c\}$ is a collection of $g+k$ pairwise-disjoint curves in $\overline{\Sigma}$ consisting of $g-k$ embedded circles $\alpha_i^c$ in the interior of $\overline{\Sigma}$ and $2k$ arcs $\alpha_j^a$ with boundary on and transverse to $\partial\overline{\Sigma}$,
			\item $\bm{\beta}=\{\beta_1,\dots,\beta_g\}$ is a collection of $g$ pairwise disjoint embedded circles $\beta_i$ in the interior of $\overline{\Sigma}$,
			\item and $z$ is a point in $\partial\overline{\Sigma}\smallsetminus(\bm{\alpha}\cap\partial\overline{\Sigma})$
		\end{itemize}
		such that $\overline{\Sigma}\smallsetminus\bm{\alpha}$ and $\overline{\Sigma}\smallsetminus\bm{\beta}$ are connected and any intersections of $\alpha$- and $\beta$ curves is transverse. Moreover, two bordered Heegaard diagrams specify the same bordered 3-manifold $Y$ if and only if they can be related to one another by a finite sequence of Heegaard moves fixing the endpoints of the $\alpha$-arcs (cf. \cite[Chapter 4]{Lipshitz2018}).
	\end{definition}
	In \cite{Lipshitz2018}, Lipshitz--Ozsv\'{a}th--Thurston associate to a bordered Heegaard diagram $\mathcal{H}=(\overline{\Sigma},\bm{\alpha},\bm{\beta},z)$ for $Y$ a left type-$D$ structure $\widehat{\CFD}(Y):=\widehat{\CFD}(\mathcal{H})$, over an algebra $\mathcal{A}=\mathcal{A}(-\mathcal{Z})$ associated to the pointed matched circle $\mathcal{Z}$ with its opposite orientation, whose homotopy type is an invariant of the bordered 3-manifold $Y$. We briefly recall this construction here.
	\begin{definition}
		Let $\mathcal{H}=(\overline{\Sigma},\bm{\alpha},\bm{\beta},z)$ be a genus $g$ bordered Heegaard diagram for a 3-manifold $Y$. A \emph{generator} of $\mathcal{H}$ is a $g$-element set $\bm{x}=\{x_1,\dots,x_g\}$ consisting of intersection points of $\bm{\alpha}$- and $\bm{\beta}$-curves such that
		\begin{itemize}
			\item exactly one $x_i$ lies (1) on each $\alpha$-circle and (2) on each $\beta$-circle, and
			\item at most one $x_i$ lies on each $\alpha$-arc.
		\end{itemize}
		Denote the set of generators of $\mathcal{H}$ by $\mathfrak{S}(\mathcal{H})$.
	\end{definition}
	The $\F_2$-vector space $\F_2\mathfrak{S}(\mathcal{H})$ spanned by $\mathfrak{S}(\mathcal{H})$ admits a left-action by the subalgebra $\mathcal{I}\subset\mathcal{A}$ of idempotents. As a left $\mathcal{A}$-module, $\widehat{\CFD}(\mathcal{H})$ is then defined by $\widehat{\CFD}(\mathcal{H}):=\mathcal{A}\otimes_{\mathcal{I}}\F_2\mathfrak{S}(\mathcal{H})$ and it is endowed with the structure of a left differential module by taking
	\begin{align}
		\partial(\id\otimes\bm{x})=\sum_{\bm{y}\in\mathfrak{S}(\mathcal{H})}\sum_{B\in\pi_2(\bm{x},\bm{y})}a_{\bm{x},\bm{y}}^B\otimes\bm{y},
	\end{align}
	where
	\begin{align}
		a_{\bm{x},\bm{y}}^B=\sum_{\vec{\rho}\,|\,\mathrm{ind}(B,\vec{\rho})=1}\#\mathcal{M}^B(\bm{x},\bm{y};\vec{\rho})a(-\vec{\rho}).
	\end{align}
	Here, $\pi_2(\bm{x},\bm{y})$ is the space of homology classes of Whitney disks connecting $\bm{x}$ and $\bm{y}$, $\vec{\rho}=(\rho_1,\dots,\rho_n)$ is a sequence of Reeb chords in $\partial\overline{\Sigma}$, $\mathcal{M}^B(\bm{x},\bm{y};\vec{\rho})$ is the moduli space of pseudoholomorphic representatives of $B$ in $\mathrm{int}(\overline{\Sigma})\times[0,1]\times\R$ with asymptotic condition $\vec{\rho}$ at east infinity modulo translation, and $a(-\vec{\rho})=a(-\rho_1)\cdots a(-\rho_n)$ is the product of algebra elements associated to the tuple $-\vec{\rho}=(-\rho_1,\dots,-\rho_n)$ of Reeb chords with their orientations reversed. The quantity $\ind(B,\vec{\rho})$ is a Maslov-type index which guarantees that the moduli space $\mathcal{M}^B(\bm{x},\bm{y};\vec{\rho})$ is 0-dimensional when $\ind(B,\vec{\rho})=1$. More generally, the moduli spaces $\mathcal{M}^B(\bm{x}_1,\dots,\bm{x}_n;\vec{\rho})$ of pseudoholomorphic $n$-gons in $\mathrm{int}(\overline{\Sigma})\times\mathbb{D}_n^2$ for $n>2$, where $\mathbb{D}_n^2$ is an $n$-times boundary-punctured disk, is 0-dimensional whenever $\ind(B,\vec{\rho})=0$.
	
	In \cite{Lipshitz_2011}, Lipshitz--Ozsv\'{a}th--Thurston show that, for bordered 3-manifolds $Y_1$ and $Y_2$ with the same boundary, the chain complex $\Mor^{\mathcal{A}}(\widehat{\CFD}(Y_1),\widehat{\CFD}(Y_2))$ of $\mathcal{A}$-module maps $\widehat{\CFD}(Y_1)\to\widehat{\CFD}(Y_2)$ is homotopy equivalent to the Heegaard Floer chain complex $\widehat{\CF}(-Y_1\cup_\partial Y_2)$. There, they considered an $(\alpha,\beta)$-bordered Heegaard diagram $\AZ(\mathcal{Z})$, first introduced by Auroux in \cite{Auroux2010}, associated to $\mathcal{Z}$ and show that the bordered Floer bimodule $\widehat{\CFAA}(\AZ(\mathcal{Z}))$ is isomorphic, as a left-right $(\mathcal{A}(-\mathcal{Z}),\mathcal{A}(-\mathcal{Z}))$-bimodule, to the regular bimodule ${}_{\mathcal{A}(-\mathcal{Z})}\mathcal{A}(-\mathcal{Z})_{\mathcal{A}(-\mathcal{Z})}$. As a corollary, they then deduce that
	\begin{align}
		\begin{split}
			\widehat{\CF}(-Y_1\cup_\partial Y_2)&\simeq\overline{\widehat{\CFD}(Y_1)}\boxtimes\widehat{\CFAA}(\AZ(\mathcal{Z}))\boxtimes\widehat{\CFD}(Y_2)\\
			&\cong\Mor^{\mathcal{A}(-\mathcal{Z})}(\widehat{\CFD}(Y_1),\widehat{\CFD}(Y_2)).
		\end{split}
	\end{align}
	The diagram $\AZ(\mathcal{Z})$ is defined as follows: if $k$ is the genus of the surface $F(\mathcal{Z})$ determined by $\mathcal{Z}$, consider the planar triangle $\triangle_k$ bounded by the coordinate axes and the line $x+y=4k+1$, which we will call the \emph{diagonal} of $\triangle_k$. Let $\Sigma'$ be the quotient of $\triangle_k$ by identifying small neighborhoods of the points $(i,4k+1-i)$ and $(j,4k+1-j)$ in the diagonal if $i$ and $j$ are matched in $\mathcal{Z}$ in such a way that the result is an orientable genus $k$ surface with a single boundary component. If $i$ and $j$ are matched in $\mathcal{Z}$, then the disconnected subspace $\triangle_k\cap(\{x=i\}\cup\{x=j\})$ descends to a single arc $\beta_i$ in $\Sigma'$ and, similarly, the subspace $\triangle_k\cap(\{y=4k+1-i\}\cup\{y=4k+1-j\})$ descends to a single arc $\alpha_i$. Let $\Sigma$ be the result of attaching a 1-handle to $\partial\Sigma'$ along the 0-sphere $\{(0,0),(4k+1,0)\}$ and let $\bm{z}$ be a neighborhood of the core of this 1-handle. Then $\AZ(\mathcal{Z})$ is the diagram $(\Sigma,\bm{\alpha},\bm{\beta},\bm{z})$, where $\bm{\alpha}=\{\alpha_i\}$ and $\bm{\beta}=\{\beta_i\}$.
	
	\begin{figure}
		\begin{center}
			\includegraphics[scale=1]{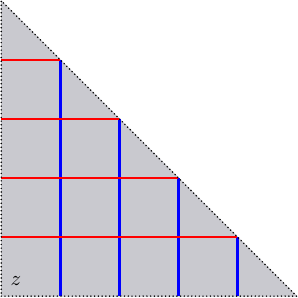}
		\end{center}
		\caption{The triangle $\triangle_1$ and the arcs which descend to the $\alpha$- and $\beta$-arcs in the interpolating piece $\AZ(\mathcal{Z})$ associated to the unique genus 1 pointed matched circle.}
	\end{figure}
	\begin{figure}
		\begin{center}
			\includegraphics[scale=1]{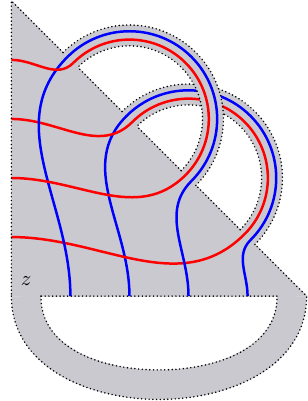}
		\end{center}
		\caption{The diagram $\AZ(\mathcal{Z})$ associated to the genus 1 pointed matched circle.}
	\end{figure}
	\section{An interpolating triple}
	We will consider a similarly defined bordered Heegaard triple associated to $\mathcal{Z}$, also due to Auroux, which we call $\AT(\mathcal{Z})$. We construct $\AT(\mathcal{Z})$ as follows: if, as before, $k$ is the genus of $F(\mathcal{Z})$, consider the square $\square_k$ in the plane bounded by the coordinate axes and the lines $x=4k+1$ and $y=4k+1$ and let $\Sigma'$ be the quotient of $\square_k$ obtained by identifying small neighborhoods of the points $(i,4k+1)$ and $(j,4k+1)$ in the segment $\square_k\cap\{y=4k+1\}$ if $i$ and $j$ are matched in $\mathcal{Z}$ in such a way that the result is an orientable genus $k$ surface with one boundary component. Now, if $i$ and $j$ are matched in $\mathcal{Z}$, then the disconnected subspaces
	\begin{align}
		\begin{split}
			g_i&=\square_k\cap(\{-x+y=4k+1-i\}\cup\{-x+y=4k+1-j\})\\
			d_i&=\square_k\cap(\{x=i\}\cup\{x=j\})\\
			e_i&=\square_k\cap(\{x+y=4k+1-i\}\cup\{x+y=4k+1-j\})
		\end{split}
	\end{align}
	descend to single arcs $\gamma_i'$, $\delta_i'$, and $\varepsilon_i'$, respectively, in $\Sigma'$. Now let $\overline{\Sigma}_{\AT}$ be the result of attaching 1-handles to $\partial\Sigma'$ along the 0-spheres $\{(0,0),(4k+1,0)\}$ and $\{(4k+1,0),(4k+1)\}$ and let $\bm{z}$ be a neighborhood of the core of either handle and take $\AT(\mathcal{Z})$ to be the triple $(\overline{\Sigma}_{\AT},\bm{\gamma},\bm{\delta},\bm{\varepsilon},\bm{z})$, where, as before, $\bm{\gamma}=\{\gamma_i\}$, $\bm{\delta}=\{\delta_i\}$, and $\bm{\varepsilon}=\{\varepsilon_i\}$ are given by suitably generic Hamiltonian perturbations of the arcs $\gamma_i'$, $\delta_i'$, and $\varepsilon_i'$. Note that the unperturbed arcs have nongeneric triple intersections so the perturbations are strictly necessary in order for the result to be an admissible diagram in the sense of \cite{Lipshitz2018}. We will perturb the triple intersections, in the same manner as given by Auroux in \cite{Auroux2010}, as shown in Figure \ref{Perturbation}. We also include in $\AT$ the data of an embedded trivalent tree $\bm{z}$ as shown in Figure \ref{Square}; in the quotient $\AT$, this tree has one leaf on each boundary component.
	
	Since it will be convenient for us to have done so later, we will modify $\AT$ slightly by assuming that the spaces $g_i$ and $e_i$ are given by lines of slope $\tan(\frac{\pi}{6})$ and $\tan(\frac{5\pi}{6})$, respectively, instead of $1$ and $-1$. We assume these again meet the top boundary segment of $\square_k$ at the points $(i,4k+1)$. If we think of these lines as the intersections of lines in $\R^2$ with $\square_k$, then the perturbations of the curves in $\AT$ which removes the nongeneric triple points can be realized by translations of the $g$- and $e$-lines in the plane as shown in Figure \ref{AT1}. This choice is motivated by the proof of Lemma \ref{TriangleLemma}.
	
	\begin{figure}\label{Perturbation}
		\begin{center}
			\includegraphics[scale=1]{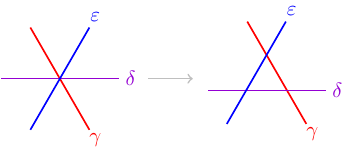}
		\end{center}
		\caption{Auroux's perturbation convention for triple intersections in $\AT(\mathcal{Z})$.}
	\end{figure}
	
	\begin{figure}
		\begin{center}
			\includegraphics[scale=1]{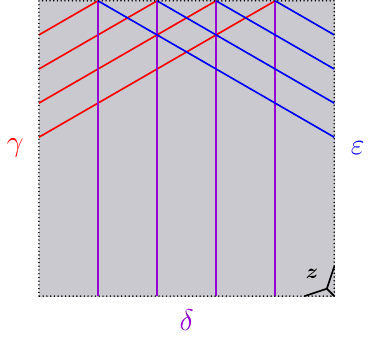}
		\end{center}
		\caption{The square $\square_1$ and the arcs which descend to the $\gamma$-, $\delta$-, and $\varepsilon$-arcs in the interpolating triple $\AT(\mathcal{Z})$ associated to the unique genus 1 pointed matched circle.}
		\label{Square}
	\end{figure}
	\begin{figure}
		\begin{center}
			\includegraphics[scale=0.875]{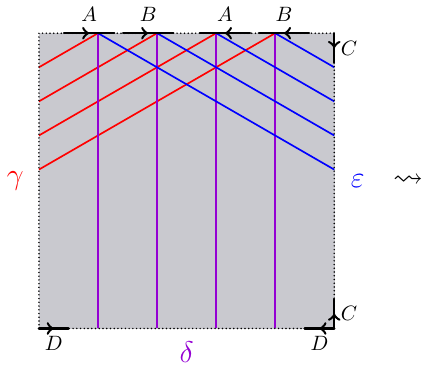}\hspace{0.125cm} \includegraphics[scale=0.875]{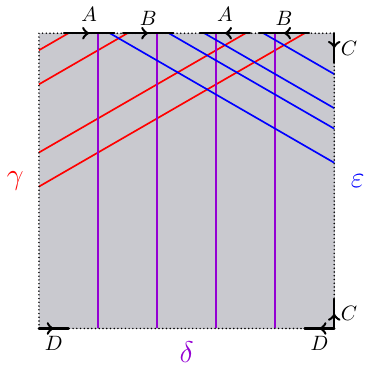}
		\end{center}
		\caption{Perturbing the diagram using planar translations to obtain the triple $\AT(\mathcal{Z})$ (right) associated to the genus 1 pointed matched circle. Here, we draw the segments of $\partial\square_k$ which are identified in $\AT$ as oriented black lines and label the glued pairs of segments with the same letter.\label{AT1}}
	\end{figure}
	Now let $\bm{\eta}$ be any one of $\bm{\gamma}$, $\bm{\delta}$, or $\bm{\varepsilon}$ and  let $\partial_\eta\AT(\mathcal{Z})$ be the component of $\partial\AT(\mathcal{Z})$ which intersects $\bm{\eta}$ nontrivially. Note that, by construction, the result of forgetting $\bm{\eta}$ and gluing a disk to $\Sigma$ along $\partial_\eta\AT(\mathcal{Z})$ is a copy of $\AZ(\mathcal{Z})$. For $\bm{\eta},\bm{\theta}\in\{\bm{\gamma},\bm{\delta},\bm{\varepsilon}\}$, let $\AZ_{\eta\theta}$ be the diagram obtained by deleting the collection of arcs $\bm{\zeta}\in\{\bm{\gamma},\bm{\delta},\bm{\varepsilon}\}\smallsetminus\{\bm{\eta},\bm{\theta}\}$ and let $\mathcal{A}_{\eta\theta}=\widehat{\CFAA}(\AZ_{\eta\theta})$. We recall \cite[Proposition 4.8]{Auroux2010} which says that the map $\mathcal{A}_{\delta\varepsilon}\otimes_\F\mathcal{A}_{\gamma\delta}\to\mathcal{A}_{\gamma\varepsilon}$ given by counting provincial holomorphic triangles in $\AT(\mathcal{Z})$ coincides with multiplication under the identification of $\mathcal{A}_{\eta\theta}$ with $\mathcal{A}(\mathcal{Z})$.
	\begin{proposition}[\cite{Lipshitz_2011}, Proposition 4.1]
		The left-right $(\mathcal{A}(\mathcal{Z}),\mathcal{A}(\mathcal{Z}))$-bimodule $\widehat{\CFAA}(\AZ_{\eta\theta})$ is isomorphic to $\mathcal{A}(\mathcal{Z})$.
	\end{proposition}
	\begin{proof}[Sketch]
		We identify the generating set $\mathfrak{S}(\AZ_{\eta\theta})$ with the usual basis for $\mathcal{A}(\mathcal{Z})$ in terms of strand diagrams. A generator $\bm{x}\in\mathfrak{S}(\AZ_{\eta\theta})$ is a set of points in $\bm{\eta}\cap\bm{\theta}$. To a single intersection point $x\in\bm{\eta}\cap\bm{\theta}$, we associate a Reeb chord or smeared horizontal strand in $\mathcal{Z}=(Z,\bm{a},M)$ as follows. First, draw $\mathcal{Z}$ above the square, oriented from left to right, with the set of points $\bm{a}$ identified with the boundary intersection points of $\bm{\eta}$ and $\bm{\theta}$. Next, note that there are unique segments $e$ and $g$ in the square passing through $x$ and there is an unique triangular (or empty) region $T_x$ of $\square_k$ bounded by the segments $e$ and $g$ and the line $y=4k+1$. If $T_x$ is empty, then $x$ is a boundary intersection point and we associate to it the smeared horizontal strand given by the matching $M$. Otherwise, we associate to $x$ the Reeb chord $\rho_x$ in $\mathcal{Z}$ determined by the line segment $T_x\cap\{y=4k+1\}$. A generator $\bm{x}\in\mathfrak{S}(\AZ_{\eta\theta})$ may therefore be identified with a set of Reeb chords and smeared horizontal strands and, hence, with a strand diagram. It is straightforward to see that this identification gives a bijection between $\mathfrak{S}(\AZ_{\eta\theta})$ and the usual basis for $\mathcal{A}(\mathcal{Z})$. Note also that we may identify the left- and right-idempotents of a generator $\bm{x}$ with the collections of left- and right-endpoints of the segments $T_x\cap\{y=4k+1\}$, respectively. The identification we have given here is equivalent to the one given in \cite{Lipshitz_2011}. To recover theirs from ours, note that if $T_x$ is nonempty, then there is an unique rectangular domain $R_x$ in $\AZ_{\eta\theta}$ bounded by the leftmost segment of $\partial\square_k$, $\bm{\eta}$, and $T_x\cap\bm{\theta}$ with vertices at $x$ and the topmost endpoint of $T_x\cap\bm{\theta}$. Drawing $\mathcal{Z}$ oriented downward and to the left of $\AZ_{\eta\theta}$ so that $\bm{a}$ is identified with $\bm{\eta}\cap\partial\AZ_{\eta\theta}$, one can verify readily that the Reeb chord in $\mathcal{Z}$ determined by $R_x\cap\partial\AZ_{\eta\theta}$ is precisely $\rho_x$. Lastly, the diagram $\AZ_{\eta\theta}$ is \emph{nice} in the sense of \cite{Sarkar2010} so the differential on $\widehat{\CFAA}(\AZ_{\eta\theta})$ counts only embedded rectangles, the only nontrivial $A_\infty$-operations are the $m_2$ maps, and these operations count half-strips --- i.e. bigons asymptotic to Reeb chords at the boundary. It is then straightforward to identify the differential and bimodule structures on $\widehat{\CFAA}(\AZ_{\eta\theta})$ with those on $\mathcal{A}(\mathcal{Z})$.
		\begin{figure}
			\begin{center}
				\includegraphics[scale=1]{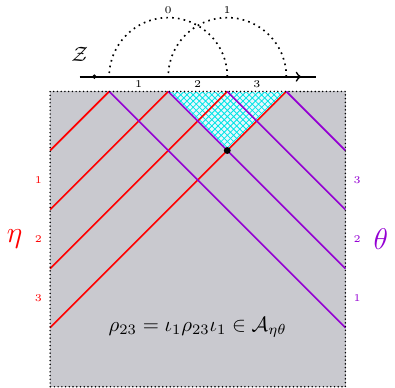}
			\end{center}
			\caption{Identifying a generator $\bm{x}\in\mathfrak{S}(\AZ_{\eta\theta})$ with the algebra element $\rho_{\scriptscriptstyle23}\in\mathcal{A}(\mathcal{Z})$. In $\mathcal{A}(-\mathcal{Z})$, this same generator is identified with $\rho_{\scriptscriptstyle12}$.}
		\end{figure}
	\end{proof}
	\begin{remark}
		One way to think about the module actions on $\widehat{\CFAA}(\AZ_{\eta\theta})$ is as follows. Suppose $\bm{x}$ and $\bm{y}$ are generators such that the collection of right-endpoints of the segments $T_{x_i}\cap\{y=4k+1\}$ for $x_i\in\bm{x}=\{x_1,\dots,x_k\}$ coincides with the collection of left-endpoints of the segments $T_{y_j}\cap\{y=4k+1\}$ for $y_j\in\bm{y}=\{y_1,\dots,y_k\}$. In this case, there is a bijection $f:[k]\to[k]$ with the property that $T_{x_i}\cap T_{y_{f(j)}}$ is precisely the common vertex of the triangles $T_{x_i}$ and $T_{y_{f(i)}}$ when $i=j$ and empty otherwise. Note that there is an unique (possibly empty) rectangular region $R_i$ with the property that $T_{z_i}:=T_{x_i}\cup T_{y_{f(i)}}\cup R_i$ is again a triangle. The product $\bm{x}\cdot\bm{y}$ is then precisely the collection of intersection points $\bm{z}=\{z_1,\dots,z_k\}$. One may verify that this coincides with the usual algebra structure on $\mathcal{A}(\mathcal{Z})$ under the above identification and with the left- and right-module structures under the identification from \cite{Lipshitz_2011}.
	\end{remark}
	We now define the map $m:\mathcal{A}_{\delta\varepsilon}\otimes_\F\mathcal{A}_{\gamma\delta}\to\mathcal{A}_{\gamma\varepsilon}$.  Let $\triangle$ be a triangle with edges $e_\gamma$, $e_\delta$, and $e_\varepsilon$, ordered clockwise, and let $e_{\eta\theta}$ be the unique point in $e_\eta\cap e_\theta$. Now let $W=\mathrm{int}(\AT)\times\triangle$ and fix generators $\rho\in\mathfrak{S}(\AZ_{\delta\varepsilon})$, $\sigma\in\mathfrak{S}(\AZ_{\gamma\delta})$, and $\tau\in\mathfrak{S}(\AZ_{\gamma\varepsilon})$. Denote by $\pi_2(\rho,\sigma,\tau)$ the collection of all homology classes of maps $(S,\partial S)\to(W,\bm{\gamma}\times e_\gamma\cup\bm{\delta}\times e_\delta\cup\bm{\varepsilon}\times e_\varepsilon)$, where $S$ is a Riemann surface with boundary and boundary marked points $s_{\gamma\delta}$, $s_{\delta\varepsilon}$, and $s_{\varepsilon\gamma}$ such that $s_{\gamma\delta}\mapsto\rho$, $s_{\delta\varepsilon}\mapsto\sigma$, and $s_{\varepsilon\gamma}\mapsto\tau$. As in Section 10 of \cite{Lipshitz2006}, one may pick a sufficiently nice almost complex structure $J$ on $W$ so that, for each $A\in\pi_2(\rho,\sigma,\tau)$, the moduli space $\mathcal{M}^A(\rho,\sigma,\tau)$ of embedded $J$-holomorphic curves $(S,\partial S)\stackrel{u}{\to}(W,\bm{\gamma}\times e_\gamma\cup\bm{\delta}\times e_\delta\cup\bm{\varepsilon}\times e_\varepsilon)$ in the homology class $A$ such that $u(s_{\gamma\delta})=\rho$, $u(s_{\delta\varepsilon})=\sigma$, and $u(s_{\varepsilon\gamma})=\tau$ is a smooth manifold whose dimension is given by the Maslov index $\mathrm{ind}(A)$ of $A$. We then define $m$ on generators by
	\begin{align}
		m(\rho\otimes\sigma)=\sum_{\tau\in\mathfrak{S}(\AZ_{\gamma\varepsilon})}\sum_{\mathrm{ind}(A)=0}\#\mathcal{M}^A(\rho,\sigma,\tau)\tau.
	\end{align}
	\begin{proposition}[{\cite[Proposition 4.8]{Auroux2010}}]
		The map $m:\mathcal{A}_{\delta\varepsilon}\otimes_\F\mathcal{A}_{\gamma\delta}\to\mathcal{A}_{\gamma\varepsilon}$ coincides with the multiplication map under the identification of each $\mathcal{A}_{\eta\theta}$ with $\mathcal{A}(\mathcal{Z})$.
	\end{proposition}
	\begin{figure}
		\begin{center}
			\includegraphics[scale=1]{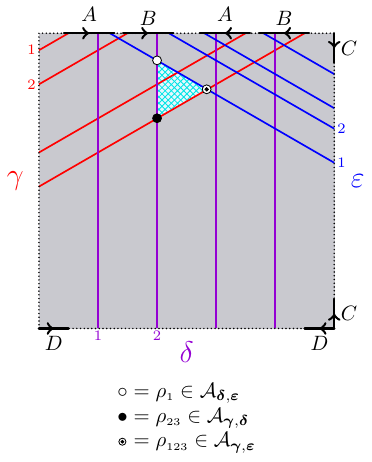}
		\end{center}
		\caption{An embedded holomorphic triangle in $\AT(\mathcal{Z})$ representing the multiplication $m^{\mathrm{op}}(\rho_{\scriptscriptstyle 23}\otimes\rho_{\scriptscriptstyle 1})=\rho_{\scriptscriptstyle 123}$ in the algebra $\mathcal{A}(\mathcal{Z})^{\mathrm{op}}$ or, equivalently, the multiplication $m(\rho_{\scriptscriptstyle12}\otimes\rho_{\scriptscriptstyle3})=\rho_{\scriptscriptstyle123}$ in $\mathcal{A}(-\mathcal{Z})$, where $\mathcal{Z}$ is the genus 1 pointed matched circle.}
	\end{figure}
	As noted in the introduction, we will be working over the algebra $\mathcal{A}(-\mathcal{Z})$. However, it is a standard fact that this algebra is isomorphic to $\mathcal{A}(\mathcal{Z})^{\mathrm{op}}$. Indeed, one can identify the generators $\mathfrak{S}(\AZ_{\eta\theta})$ with the usual generators for $\mathcal{A}(-\mathcal{Z})$ in precisely the same way as we did for $\mathcal{A}(\mathcal{Z})$ with the sole exception that we draw $\mathcal{Z}$ above $\AZ_{\eta\theta}$ oriented from right to left, rather than from left to right.
	\begin{corollary}
		$m$ coincides with the multiplication map $\mathcal{A}_{\gamma\delta}\otimes_\F\mathcal{A}_{\delta\varepsilon}\to\mathcal{A}_{\gamma\varepsilon}$ under the identification of $\mathcal{A}_{\eta\theta}$ with $\mathcal{A}(-\mathcal{Z})$.
	\end{corollary}
	\begin{remark}
		By construction, the map $m$ counts only pseudoholomorphic triangles which do not meet the boundary of $\AT$. One could instead count all rigid triangles in $\AT$, in which case one would expect to see additional terms in $m$. However, Lemma \ref{TriangleLemma} below tells us that these maps coincide. See \cite{LOT2} for further details on pseudoholomorphic polygon maps in bordered Floer homology.
	\end{remark}
	\section{Composition and triangle counts}
	\begin{definition}
		We say that $f\in\Mor^{\mathcal{A}}(Y_1,Y_2)$ is a \emph{basic morphism} if there are left-module generators $\bm{u}\in\widehat{\CFD}(Y_1)$ and $\bm{v}\in\widehat{\CFD}(Y_2)$ and an algebra element $\rho\in\mathcal{A}(-\mathcal{Z})$ such that $f(\bm{u})=\rho\bm{v}$ and $f$ vanishes on all other generators.
	\end{definition}
	\begin{lemma}
		The set of basic morphisms forms an $\F$-basis for $\Mor^{\mathcal{A}}(Y_1,Y_2)$.
	\end{lemma}
	\begin{proof}
		Let $\bm{u}_1,\dots,\bm{u}_m\in\widehat{\CFD}(Y_1)$ and $\bm{v}_1,\dots,\bm{v}_n\in\widehat{\CFD}(Y_2)$ be the generators for a given choice of bordered Heegaard diagrams for the $Y_i$. For $j=1,\dots,m$, let $f_1^j,\dots,f_{s_j}^j$ be the distinct basic morphisms for which $f_i^j(\bm{u}_j)=\rho_{ij}\bm{v}_{k(i,j)}$ is nonzero. Suppose that there is a linear dependence
		\begin{align}
			\sum_{i,j}c_{ij}f_i^j=0
		\end{align}
		between them. For a given $j$, we then have a linear dependence
		\begin{align}
			\sum_ic_{ij}f_i^j(\bm{u}_j)=\sum_ic_{ij}\rho_{ij}\bm{v}_{k(i,j)}=0
		\end{align}
		but the $\rho_{ij}\bm{v}_{k(i,j)}$ are all distinct, hence $\F$-linearly independent, since the $f_i^j$ are basic and distinct so $c_{ij}=0$ for all $i$ and $j$. Now given $g\in\Mor^{\mathcal{A}}(Y_1,Y_2)$, write
		\begin{align}
			g(\bm{u}_j)=\sum_{i}\sigma_{ij}\bm{v}_i.
		\end{align}
		For each $i$ and $j$ for which $\sigma_{ij}\bm{v}_i\neq 0$, one can then define a basic morphism $g_{i,j}$ by taking $g_{i,j}(\bm{u}_j)=\sigma_{ij}\bm{v}_i$ and $g_{i,j}(\bm{u}_k)=0$ for $k\neq j$. We then have that
		\begin{align}
			g=\sum_{i,j}g_{i,j}
		\end{align}
		by construction so the basic morphisms span $\Mor^{\mathcal{A}}(Y_1,Y_2)$.
	\end{proof}
	The identification
	\begin{align}
		\Mor^{\mathcal{A}}(Y_1,Y_2)\cong\overline{\widehat{\CFD}(Y_1)}\boxtimes\mathcal{A}(-\mathcal{Z})\boxtimes\widehat{\CFD}(Y_2)
	\end{align}
	can then be given in terms of basic morphisms as follows: suppose we have a basic morphism $f:\widehat{\CFD}(Y_1)\to\widehat{\CFD}(Y_2)$ defined by $f(\bm{u})=\rho\bm{v}$, then $f$ is sent under this isomorphism to the tensor product $\overline{\bm{u}}\boxtimes\rho\boxtimes\bm{v}$. If we have a second basic morphism $g:\widehat{\CFD}(Y_2)\to\widehat{\CFD}(Y_3)$ determined by $g(\bm{v})=\sigma\bm{w}$, then the composition $g\circ f$ is given at the level of box tensor products by
	\begin{align}
		(\overline{\bm{v}}\boxtimes\sigma\boxtimes\bm{w})\circ(\overline{\bm{u}}\boxtimes\rho\boxtimes\bm{v})=\overline{\bm{u}}\boxtimes\rho\sigma\boxtimes\bm{w},
	\end{align}
	so we we may realize the composition map $f\otimes g\mapsto g\circ f$ explicitly in terms of the multiplication operation on $\mathcal{A}(\mathcal{Z})$ as:
	\begin{align}
		\begin{split}
			&(\overline{\bm{u}}\boxtimes\rho\boxtimes\bm{v})\otimes(\overline{\bm{v}}\boxtimes\sigma\boxtimes\bm{w})\stackrel{\mathit{ev}}{\mapsto}\overline{\bm{u}}\boxtimes\rho\boxtimes\overline{\bm{v}}(\bm{v})\boxtimes\sigma\boxtimes\bm{w}\\&=\overline{\bm{u}}\boxtimes\rho\boxtimes\iota_{\bm{v}}\boxtimes\sigma\boxtimes\bm{w}\stackrel{\cong}{\mapsto}\overline{\bm{u}}\boxtimes\rho\boxtimes\sigma\boxtimes\bm{w}\stackrel{m}{\mapsto}\overline{\bm{u}}\boxtimes\rho\sigma\boxtimes\bm{w},
		\end{split}
	\end{align}
	where $\mathit{ev}:\widehat{\CFD}(Y_2)\otimes_\F\overline{\widehat{\CFD}(Y_2)}\to\mathcal{A}$ is the evaluation map $\bm{x}\otimes h\mapsto h(\bm{x})$ and the map preceding $m^{\mathcal{A}}$ is given by the isomorphism $\mathcal{A}\boxtimes\mathcal{I}\boxtimes\mathcal{A}\cong\mathcal{A}\boxtimes\mathcal{A}$. Note that this penultimate step is possible because $\bm{v}$ is a generator and the restriction of the evaluation map to the $\F$-vector subspace of $\widehat{\CFD}(Y_2)\otimes_\F\overline{\widehat{\CFD}(Y_2)}$ spanned by elements of the form $\bm{v}\otimes\overline{\bm{v}}$, where $\bm{v}$ is a generator as above, takes values in the subring $\mathcal{I}$ of idempotents of $\mathcal{A}(\mathcal{Z})$.
	\subsection{Small perturbations}
	In this subsection, we show that a small perturbation of the $\beta$-circles of a bordered Heegaard diagram $\mathcal{H}=(\overline{\Sigma},\bm{\alpha},\bm{\beta},z)$ induces an isomorphism of type-$D$ modules. Let $(\overline{\Sigma},\bm{\alpha},\bm{\beta},\bm{\gamma},z)$ be a provincially admissible bordered Heegaard triple with one boundary component such that $\bm{\beta}$ and $\bm{\gamma}$ consist entirely of circles. Then $(\overline{\Sigma},\bm{\beta},\bm{\gamma},z)$ is an admissible balanced sutured Heegaard diagram for the sutured 3-manifold $Y_{\beta\gamma}\smallsetminus B^3$ with a single boundary suture. The corresponding sutured Floer complex $\mathit{SFC}(\overline{\Sigma},\bm{\beta},\bm{\gamma},z)$ is isomorphic to the ordinary Heegaard Floer complex $\widehat{\CF}(Y_{\beta\gamma})$ (cf. \cite[Proposition 9.1]{Juhasz2006}). We may then define a type-$D$ morphism
	\begin{align*}
		\widehat{f}_{\alpha\beta\gamma}:\widehat{\CFD}(Y_{\alpha\beta})\otimes\widehat{\CF}(Y_{\beta\gamma})\to\mathcal{A}\otimes\widehat{\CFD}(Y_{\alpha\gamma})
	\end{align*}
	by
	\begin{align}
		\widehat{f}_{\alpha\beta\gamma}(\bm{x}\otimes\bm{y})=\sum_{\bm{w}\in\mathfrak{S}(\bm{\alpha},\bm{\gamma})}\sum_{B\in\pi_2(\bm{x},\bm{y},\bm{w})}a_{\bm{x},\bm{y},\bm{w}}^B\otimes\bm{w},
	\end{align}
	where
	\begin{align}
		a_{\bm{x},\bm{y},\bm{w}}^B=\sum_{\vec{\rho}\,|\,\mathrm{ind}(B,\rho)=0}\#\mathcal{M}^B(\bm{x},\bm{y},\bm{w};\vec{\rho})a(-\vec{\rho}).
	\end{align}
	Here, $\pi_2(\bm{x},\bm{y},\bm{w})$ is the space of homology classes of Whitney triangles connecting $\bm{x}$, $\bm{y}$, and $\bm{w}$, $\mathcal{M}^B(\bm{x},\bm{y},\bm{z};\vec{\rho})$ is the moduli space of pseudoholomorphic representatives of $B$ with asymptotic condition $\vec{\rho}$ at east infinity, and $a(-\vec{\rho})$ is defined as before. The fact that $\widehat{f}_{\alpha\beta\gamma}$ is a morphism of type-$D$ structures follows from a straightforward variation on the usual proof that $\partial^2=0$ for $\widehat{\CFD}(Y)$. Alternatively, it is a special case of \cite[Proposition 4.29]{LOT2}.
	
	For $\bm{\beta}^1$ a small Hamiltonian perturbation of $\bm{\beta}^0$, we will show that the map $\widehat{f}_{\alpha\beta^0\beta^1}$ induces an isomorphism $\widehat{\CFD}(Y_{\alpha\beta^0})\to\widehat{\CFD}(Y_{\alpha\beta^1})$. We recall the following standard lemma \cite[Lemma 9.10]{OzsSz2004}.
	\begin{lemma}\label{FiltrationLemma}
		Let $F:A\to B$ be a map of $\R$-filtered groups admitting a decomposition $F=F_0+\ell$ where $F_0$ is a filtration-preserving isomorphism and $\ell(\bm{x})<F_0(\bm{x})$ for all generators $\bm{x}$. Then, if the filtration on $B$ is bounded below, $F$ is an isomorphism.
	\end{lemma}
	We recall here the definition of the energy filtration on $\widehat{\CFD}(\overline{\Sigma},\bm{\alpha},\bm{\beta},z)$ from \cite[Chapter 6]{Lipshitz2018}, assuming that $(\overline{\Sigma},\bm{\alpha},\bm{\beta},z)$ is admissible. Choose an area form on $\overline{\Sigma}$. Given a $\mathit{Spin}^c$-structure $\mathfrak{s}$ on $Y$, define $\mathcal{F}:\mathfrak{S}(\overline{\Sigma},\bm{\alpha},\bm{\beta},\mathfrak{s})\to\R$ as follows: choose any generator $\bm{x}_0\in\mathfrak{S}(\overline{\Sigma},\bm{\alpha},\bm{\beta},\mathfrak{s})$ and set $\mathcal{F}(\bm{x}_0)=0$. For any other generator $\bm{x}\in\mathfrak{S}(\overline{\Sigma},\bm{\alpha},\bm{\beta},\mathfrak{s})$, choose $A_{\bm{x}_0,\bm{x}}\in\pi_2(\bm{x}_0,\bm{x})$ and let
	\begin{align}
		\mathcal{F}(\bm{x})=-\mathit{Area}(A_{\bm{x}_0,\bm{x}}).
	\end{align}
	This definition is independent of the choice of $A_{\bm{x}_0,\bm{x}}$ since $(\overline{\Sigma},\bm{\alpha},\bm{\beta},z)$ is admissible. For an algebra element $a\in\mathcal{A}$ such that $a\bm{x}\neq 0$, define $\mathcal{F}(a\bm{x})=\mathcal{F}(\bm{x})$. Then $\mathcal{F}$ induces a filtration on $\widehat{\CFD}(\overline{\Sigma},\bm{\alpha},\bm{\beta},z)$.
	
	Let $\mathcal{H}_{\alpha\beta}=(\overline{\Sigma}_g,\bm{\alpha},\bm{\beta}^0,z)$ be an admissible genus $g$ bordered Heegaard diagram. Provided $\bm{\beta}^1$ is a sufficiently small perturbation of $\bm{\beta}^0$, we may identify $\bm{x}\in\mathfrak{S}(\overline{\Sigma},\bm{\alpha},\bm{\beta}^0,\mathfrak{s})$ with its ``nearest neighbor'' $\bm{x}^1\in\mathfrak{S}(\overline{\Sigma},\bm{\alpha},\bm{\beta}^1,\mathfrak{s})$. This identification extends to a vector space isomorphism $\widehat{\CFD}(\overline{\Sigma},\bm{\alpha},\bm{\beta}^0,z)\to\widehat{\CFD}(\overline{\Sigma},\bm{\alpha},\bm{\beta}^1,z)$ --- which then extends automatically to an isomorphism $\Psi_{0\to 1}$ of type-$D$ structures.
	
	Note that if $\bm{\beta}^1$ is a small perturbation of $\bm{\beta}^0$ as above, then the homology of the complex $\widehat{\CF}(\mathcal{H}_{\beta^0\beta^1})$ associated to the diagram $\mathcal{H}_{\beta^0\beta^1}=(\overline{\Sigma}_g,\bm{\beta}^0,\bm{\beta}^1,z)$ is given by $\widehat{\HF}(\#^gS^2\times S^1)$ since $\mathcal{H}_{\beta^0\beta^1}$ is an admissible balanced sutured Heegaard diagram for $\#^gS^2\times S^1\smallsetminus B^3$.
	\begin{lemma}\label{PerturbationLemma}
		Let $\Theta^{\mathrm{top}}_{\beta^0\beta^1}$ denote the canonical top-dimensional homology class in $\widehat{\HF}(\#^gS^2\times S^1)$. Then the map $\widehat{F}_{\alpha\beta^0\beta^1}^{\mathrm{top}}:\widehat{\CFD}(\mathcal{H}_{\alpha\beta^0})\to\mathcal{A}\otimes\widehat{\CFD}(\mathcal{H}_{\alpha\beta^1})$ given by
		\begin{align}
			\bm{x}\mapsto\widehat{f}_{\alpha\beta^0\beta^1}(\bm{x}\otimes \Theta^{\mathrm{top}}_{\beta^0\beta^1})
		\end{align}
		is an isomorphism of type-$D$ structures. Moreover, this map is homotopic to the nearest point map.
	\end{lemma}
	\begin{proof}
		Let $T_{\bm{x}}\in\pi_2(\bm{x},\Theta^{\mathrm{top}}_{\beta^0\beta^1},\bm{x}^1)$ be the canonical smallest triangle, which has an unique holomorphic representative by the Riemann mapping theorem. Provided our perturbation is small enough, we may assume that the area of $T_{\bm{x}}$ is smaller than the areas of all classes in $\pi_2(\bm{x},\bm{y})$ for any generators $\bm{x}$ and $\bm{y}$ in either $\mathfrak{S}(\overline{\Sigma},\bm{\alpha},\bm{\beta}^0,\mathfrak{s})$ or $\mathfrak{S}(\overline{\Sigma},\bm{\alpha},\bm{\beta}^1,\mathfrak{s})$. Moreover, we may choose the area form so that $T_{\bm{x}}$ is the unique triangle of minimal area connecting $\bm{x}$, $\bm{y}$, and $\Theta^{\mathrm{top}}_{\beta^0\beta^1}$ among all $\bm{y}\in\mathfrak{S}(\overline{\Sigma},\bm{\alpha},\bm{\beta}^1)$. Let $\mathcal{F}_0^1$ be the filtration on $\widehat{\CFD}(\overline{\Sigma},\bm{\alpha},\bm{\beta}^1,z)$ defined as above. Define a new filtration $\mathcal{F}^1$ on $\widehat{\CFD}(\overline{\Sigma},\bm{\alpha},\bm{\beta}^1,z)$ by taking $\mathcal{F}^1(\bm{x}^1)=\mathcal{F}_0^1(\bm{x}^1)-\mathit{Area}(T_{\bm{x}_0})$. As in \cite[Proposition 6.41]{Lipshitz2018}, the map $\widehat{F}_{\alpha\beta^0\beta^1}^{\mathrm{top}}$ is filtered with respect to $\mathcal{F}$ and $\mathcal{F}^1$ and the filtration-preserving part of $\widehat{F}_{\alpha\beta^0\beta^1}^{\mathrm{top}}$ is given by $\Psi_{0\to 1}$. Note that we may promote $\widehat{F}_{\alpha\beta^0\beta^1}^{\mathrm{top}}$ and $\Psi_{0\to 1}$ to maps $\mathcal{A}\otimes\widehat{\CFD}(\overline{\Sigma},\bm{\alpha},\bm{\beta}^0,z)\to\mathcal{A}\otimes\widehat{\CFD}(\overline{\Sigma},\bm{\alpha},\bm{\beta}^1,z)$ of differential left $\mathcal{A}$-modules by taking $\widehat{F}_{\alpha\beta^0\beta^1}^{\mathrm{top}}(a\otimes\bm{x})=a\widehat{F}_{\alpha\beta^0\beta^1}^{\mathrm{top}}(\bm{x})$ and similarly for $\Psi_{0\to 1}$. Since $\Psi_{0\to 1}$ is a vector space isomorphism, it follows from Lemma \ref{FiltrationLemma} that $\widehat{F}_{\alpha\beta^0\beta^1}^{\mathrm{top}}$ is an isomorphism of  differential left $\mathcal{A}$-modules and hence of type-$D$ structures. One can easily adapt the argument given in \cite[Lemma 5.4]{Guth2022} to show that $\widehat{F}_{\alpha\beta^0\beta^1}^{\mathrm{top}}$ is homotopic to the nearest point map (cf. also \cite[Proposition 11.4]{Lipshitz2006}).
	\end{proof}
	We now recall a few definitions and results about holomorphic polygons with Reeb chord asymptotics. Denote by $D_n$ an $n$-gon, i.e. a disk with $n$ labeled punctures on its boundary. Label the boundary arcs clockwise as $e_0,\dots,e_{n-1}$ and let $p_{i,i+1}$ be the puncture between $e_i$ and $e_{i+1}$. Define $\mathrm{Conf}(D_n)$ to be the moduli space of positively-oriented complex structures on $D_n$ up to labeling-preserving biholomorphisms. Recall that this space has a Deligne-Mumford compactification $\overline{\mathrm{Conf}}(D_n)$ which is diffeomorphic to the associahedron and whose boundary $\partial\overline{\mathrm{Conf}}(D_n)$ consists of trees of equivalence classes of complex structures on polygons with each edge representing a gluing of two polygons along a vertex.
	\begin{definition}[{\cite[Definition 3.5]{LOT2}}]
		For a fixed symplectic form $\omega_\Sigma$ on a Riemann surface $\Sigma$, an \emph{admissible collection of almost-complex structures} is a choice of $\R$-invariant almost complex structure $J$ on $\Sigma\times[0,1]\times\R$ and a smooth family $\{J_j\}_{j\in\mathrm{Conf}(D_n)}$ of almost complex structures on $\Sigma\times D_n$ for each $n\geq 3$ such that the following conditions hold:
		\begin{itemize}
			\item For each $j\in\mathrm{Conf}(D_n)$, the projection $\pi_{\mathbb{D}}:\Sigma\times D_n\to D_n$ is $(J_j,j)$-holomorphic.
			\item For every $j\in\mathrm{Conf}(D_n)$, the fibers of $\pi_{\mathbb{D}}$ are $J_j$-holomorphic.
			\item Every $J_j$ is adjusted to the split symplectic form $\omega_\Sigma\oplus\omega_j$ on $\Sigma\times D_n$.
			\item  Each $J_j$ agrees with $J$ near the punctures of $D_n$ in the sense that every puncture has a strip-like neighborhood $U$ in $D_n$ such that $(\Sigma\times U,J_j|_{\Sigma\times U})$ and $(\Sigma\times[0,1]\times(0,\infty),J)$ are biholomorphically equivalent.
			\item If $(j_k)$ is a sequence in $\mathrm{Conf}(D_n)$ converging to some point $j_\infty\in\partial\overline{\mathrm{Conf}}(D_n)$ lying in the codimension-$1$ boundary stratum, i.e. a point $(j_{\infty,1},j_{\infty,2})\in\mathrm{Conf}(D_{m+1})\times\mathrm{Conf}(D_{n-m+1})$ for some $m$, then the complex structures $J_{j_k}$ converge to $J_{j_{\infty,1}}\sqcup J_{j_{\infty,2}}$ on $(\Sigma\times D_{m+1})\sqcup(\Sigma\times D_{n-m+1})$. Convergence here is in the sense that, as $k\to\infty$, some arcs in $D_{m+1}$ collapse and, over neighborhoods of these arcs, the complex structures $J_{j_k}$ are obtained by inserting longer and longer necks the $J_{j_k}$ converge in the $C^\infty$-topology outside of these neighborhoods. The analogous compatibility condition is required for points lying in higher codimension boundary strata.
		\end{itemize}
	\end{definition}
	\begin{definition}[{\cite[Definition 4.5]{LOT2}}]
		Let $(\Sigma,\bm{\alpha},\bm{\beta}^1,\dots,\bm{\beta}^n,z)$ be an admissible bordered Heegaard multidiagram in the sense of \cite[Definition 4.2]{LOT2}, where $\bm{\alpha}$ is a complete set of bordered attaching curves compatible with $\mathcal{Z}$. Let $S$ be a punctured Riemann surface and $\{J_{j}\}_{j\in\mathrm{Conf}(D_{n+1})}$ be an admissible collection of almost complex structures. Fix generators $\bm{x}^k\in\mathfrak{S}(\bm{\beta}^k,\bm{\beta}^{k+1})$ for $k=1,\dots,n-1$ and $\bm{x}^0\in\mathfrak{S}(\bm{\alpha},\bm{\beta}^1)$, $\bm{x}^n\in\mathfrak{S}(\bm{\alpha},\bm{\beta}^n)$, and let $q_i\in\partial D_{n+1}$ be points for $i=1,\dots,k$. Consider maps of the form
		\begin{align}
			u:(S,\partial S)\to(\Sigma\times D_{n+1},(\bm{\alpha}\times e_0)\cup(\bm{\beta}^1\times e_1)\cup\cdots\cup(\bm{\beta}^n\times e_n))
		\end{align}
		such that the following hold:
		\begin{itemize}
			\item The projection map $\pi_\Sigma\circ u:S\to\Sigma$ has degree 0 at the region adjacent to the basepoint $z$.
			\item The punctures of $S$ are mapped to the punctures $\{p_{i,i+1}\}\cup\{q_i\}$ of $D_{n+1}\setminus\{q_i\}$.
			\item The map $u$ is asymptotic to $\bm{x}^i\times\{p_{i,i+1}\}$ at the preimage of $p_{i,i+1}$.
			\item $u$ is asymptotic to $\bm{\rho}_i\times\{q_i\}$ at the punctures lying above $q_i$ for some set $\bm{\rho}_i$ of Reeb chords in $\mathcal{Z}$.
			\item At each $q\in e_0\smallsetminus\{q_i\}$, the $g$ points $(\pi_\Sigma\circ u)((\pi_{\mathbb{D}}\circ u)^{-1}(q))$ lie in $g$ distinct $\alpha$-curves. Equivalently, $\bm{x}\otimes a(\bm{\rho}_1)\otimes\cdots\otimes a(\bm{\rho}_m)$ is nonzero, where tensor products are taken over the ring of idempotents in $\mathcal{A}(\mathcal{Z})$.
		\end{itemize}
		The set of maps of this type decomposes according to homology classes, the set of which we denote by $\pi_2(\bm{x}^n,\bm{x}^{n-1},\dots,\bm{x}^0;\bm{\rho}_1,\dots,\bm{\rho}_m)$. For a fixed homology class $B\in\pi_2(\bm{x}^n,\bm{x}^{n-1},\dots,\bm{x}^0;\bm{\rho}_1,\dots,\bm{\rho}_m)$, let
		\begin{align}
			\mathcal{M}^B(\bm{x}^n,\bm{x}^{n-1},\dots,\bm{x}^0;\bm{\rho}_1,\dots,\bm{\rho}_m;S)
		\end{align}
		denote the moduli space of pairs of the form $(j,u)$ with $j\in\mathrm{Conf}(D_{n+1})$ and $u$ a $J_j$-holomorphic representative of $B$.
	\end{definition}
	\begin{lemma}[{\cite[Lemma 4.7]{LOT2}}]
		The expected dimension of the moduli space $\mathcal{M}^B(\bm{x}^n,\bm{x}^{n-1},\dots,\bm{x}^0;\bm{\rho}_1,\dots,\bm{\rho}_m;S)$ is given by $\ind(B,S;\bm{\rho}_1,\dots,\bm{\rho}_m)+n-2$, where
		\begin{align}
			\ind(B,S;\bm{\rho}_1,\dots,\bm{\rho}_m)=\left(\frac{3-n}{2}\right)g-\chi(S)+2e(B)+m,
		\end{align}
		where $g$ is the genus of $\Sigma$ and $e(B)$ is the Euler measure of $B$.
	\end{lemma}
	\begin{remark}
		The same statement holds if the multidiagram has more than one boundary component, each of which meets exactly one set of bordered attaching curves.
	\end{remark}
	The Euler measure $e(B)$ can be characterized as follows: if $D$ is a surface with boundary and corners equipped with a metric $h$ such that $\partial D$ is geodesic and has right-angled corners, then $e(D)$ is $\frac{1}{2\pi}$ times the integral over $D$ of the curvature of $h$. From this definition, one can see that $e(D)$ is linear with respect to disjoint union and gluing along boundary segments so, if $B$ is a formal sum $B=\sum_i n_iD_i$ of elementary domains $D_i$, then $e(B)=\sum_i n_ie(D_i)$. It follows from the Gau\ss{}--Bonnet theorem that if $D$ is a surface as above with $k$ corners with angle $\frac{\pi}{2}$ and $\ell$ with angle $\frac{3\pi}{2}$, then
	\begin{align}
		e(D)=\chi(D)-\frac{k}{4}+\frac{\ell}{4}.
	\end{align}
	In particular, for a $k$-gon $D$ with convex corners, we have $e(D)=1-\frac{k}{4}$. Now suppose that $h$ is instead an arbitrary metric on $D$ and that $\partial D$ decomposes as $\partial D=c_1\cup\cdots\cup c_k$. Parametrize each boundary segment $c_i$ by $[0,1]$. For each $i=1,\dots,k$, let $\theta_i$
	be the angle by which the tangent vector to $\partial D$ turns at the $i^{\mathrm{th}}$ corner $c_i(0)$, i.e. $\pi$ minus the interior angle of $D$ at $c_i(0)$, and define $t_i=\frac{\theta_i}{2\pi}-\frac{1}{4}$. A second application of the Gau\ss{}--Bonnet theorem allows us to rewrite $e(D)$ as
	\begin{align}
		e(D)=\frac{1}{2\pi}\left(\int_DKdA+\sum_{i=1}^{k}\int_{c_i}\kappa_hds\right)+\sum_{i=1}^{k}t_i,
	\end{align}
	where $K$ and $\kappa_h$ are the curvature and geodesic curvature of $h$, respectively. Therefore, if $h$ is flat and $D$ has geodesic boundary, we may then compute $e(D)$ by summing the contributions $t_i$ from each corner. In particular, corners with interior angles of 60-, 90-, and 120-degrees contribute $+\frac{1}{12}$, $0$, and $-\frac{1}{12}$, respectively, to the Euler measure of a flat polygon with geodesic boundary. We will use this fact momentarily.
	
	In the case of triangles we have $n=2$ so the dimension of the moduli space $\mathcal{M}^B(\bm{x}^2,\bm{x}^{1},\bm{x}^0;\bm{\rho}_1,\dots,\bm{\rho}_m;S)$ is given exactly by $\ind(B,S;\bm{\rho}_1,\dots,\bm{\rho}_m)$, which we may write more succinctly as
	\begin{align}
		\ind(B,S;\bm{\rho}_1,\dots,\bm{\rho}_m)=\frac{g}{2}-\chi(S)+2e(B)+m.
	\end{align}
	\begin{lemma}\label{TriangleLemma}
		There are no positive domains for index zero holomorphic triangles in $\AT$ meeting $\partial\AT$ and having corners cyclically ordered according to $(\gamma\cap\delta,\delta\cap\varepsilon,\gamma\cap\varepsilon)$.
	\end{lemma}
	\begin{proof}
		We choose a metric on $\AT$ which is flat everywhere except on the component of $\AT\smallsetminus(\bm{\gamma}\cup\bm{\delta}\cup\bm{\varepsilon})$ containing $\bm{z}$. Moreover we choose this metric so that every $\gamma$-, $\delta$, and $\varepsilon$-curve is geodesic and every intersection of two such curves occurs at 60 and 120 degree angles, the boundary components of $\AT$ are geodesic, and, for every $\eta\in\{\gamma,\delta,\varepsilon\}$, each $\eta$-curve meets $\partial\AT$ at the same angle: 120 degrees for the $\gamma$-curves, 90 degrees for the $\delta$-curves, and 30 degrees for the $\varepsilon$-curves. To see that we can choose such a metric, note that the square $\square_k$ inherits a metric from its inclusion into the plane which descends to a metric on $\AT$ which is flat except on the region containing $\bm{z}$. Since the boundary of $\square_k$ is geodesic, it follows that $\partial\AT$ is geodesic. To see that every $\gamma$-, $\delta$-, and $\varepsilon$-curve is geodesic and have the specified intersection angles, recall that we chose a particular modification of $\AT$ so that these curves arise from pairs $g_i$, $d_i$, and $e_i$ of straight lines making an angle of 150 degrees, 90 degrees, and 30 degrees with the positive horizontal direction, respectively. Since the perturbations necessary to obtain the curves in $\AT$ can be achieved by planar translations of the lines in $\R^2$ corresponding to the pairs $g_i$ and $e_i$, it follows that the $\gamma$-, $\delta$-, and $\varepsilon$-curves are obtained as quotients of pairs of straight line segments with the same angle and hence are geodesic. The choice of angles of these segments guarantees that each of the intersections in $\AT$ occurs in one of the specified angles.
		
		Suppose that $B$ is a positive domain for an index zero holomorphic triangle in $\AT$ which has the above cyclic ordering on its corners and which does not meet the component of $\AT\smallsetminus(\bm{\gamma}\cup\bm{\delta}\cup\bm{\varepsilon})$ containing $\bm{z}$. As in the proof of \cite[Proposition 3.5]{Auroux2010}, the Euler measure of $B$ can be computed by summing the contributions from its corners because $\partial B$ is geodesic: $+\frac{1}{12}$ for every corner with a 60-degree angle, $0$ for every corner with a 90-degree angle, and $-\frac{1}{12}$ for every corner with a 120-degree angle. If $p$ is an interior intersection point of two of the collections of curves in $\AT$ and $B$ hits $p$ at an interior point, then the local multiplicities of $B$ in the four elementary domains meeting $p$ are all equal so the local contribution of $p$ to the Euler measure is zero. If $B$ hits $p$ at a point on the boundary which is not a corner, then $B$ hits two of the four regions meeting at $p$. One of these regions meets $p$ at a 60-degree angle and the other meets it at a 120-degree angle so the local contributions to the Euler measure cancel. If $p$ is a genuine corner of $B$, then the cyclic ordering of the corners forces one of two scenarios: either $B$ locally hits a region with a 60-degree angle at $p$ or $B$ locally hits two regions with a 60-degree angle at $p$ and one with a 120-degree angle at $p$. In either of these two cases, the local contribution of such a corner is $+\frac{1}{12}$.
		
		Now, if $p\in\bm{\eta}\cap\partial\AT$ for some $\eta\in\{\gamma,\delta,\varepsilon\}$, then there are two cases that we need to account for. Suppose, for the moment, that $B$ meets exactly one Reeb chord $\rho$ in the $\eta$-boundary of $\AT$. If $p$ is contained in the interior of $\rho$, then the local multiplicities of $B$ in the two regions meeting $p$ are equal so the local contribution to the Euler measure is zero. Otherwise, $p$ is an end of $\rho$, in which case there is a boundary intersection point $q$ with $\partial\rho=\{p,q\}$ and the local contributions of these two corners to the Euler measure cancel since $B$ meets $p$ and $q$ at complementary angles. In general, $B$ could meet multiple boundary Reeb chords in which case the sum of the local contributions of the ends of all of the Reeb chords is zero since we can decompose this as a sum of single Reeb chord terms.
		
		Summing over the $3g$ interior corners and all of the boundary Reeb chords of $B$, we see that $e(B)=\frac{g}{4}$ so, consequently, we have
		\begin{align}
			\ind(B,S;\bm{\rho}_1,\dots,\bm{\rho}_m)=g-\chi(S)+m.
		\end{align}
		For rigid triangles, this then tells us that $\chi(S)=g+m$ but $S$ has at most $g$ connected components so $\chi(S)\leq g$. Therefore, if $B$ is a class represented by a rigid holomorphic triangle, then we must have $m=0$, i.e. $B$ does not meet the boundary of $\AT$.
	\end{proof}
	Let $\mathcal{H}_i=(\overline{\Sigma}_i,\bm{\eta}_i,\bm{\beta}_i,z)$ be admissible bordered Heegaard diagrams for $Y_i$, $i=1,2,3$, where $\eta_i=\gamma,\delta,\varepsilon$ according to the ordering $\gamma<\delta<\varepsilon$. Let $\mathcal{H}_i^+=(\overline{\Sigma}_i,\bm{\eta}_i,\bm{\beta}_i^0,\bm{\beta}_i^1,z)$ be the result of creating a single parallel copy of each $\bm{\beta}$-circle and performing a finger move to create two intersection points between the resulting parallel pairs. Finally, let $\AT_{1,2,3}=\AT(\mathcal{H}_1,\mathcal{H}_2,\mathcal{H}_3)$ be the result of gluing $\mathcal{H}_1^+$, $\mathcal{H}_2^+$, and $\mathcal{H}_3^+$ along the $\bm{\gamma}$-, $\bm{\delta}$-, and $\bm{\varepsilon}$-boundaries of $\AT(\mathcal{Z})$.
	\begin{figure}
		\begin{center}
			\includegraphics[scale=0.8]{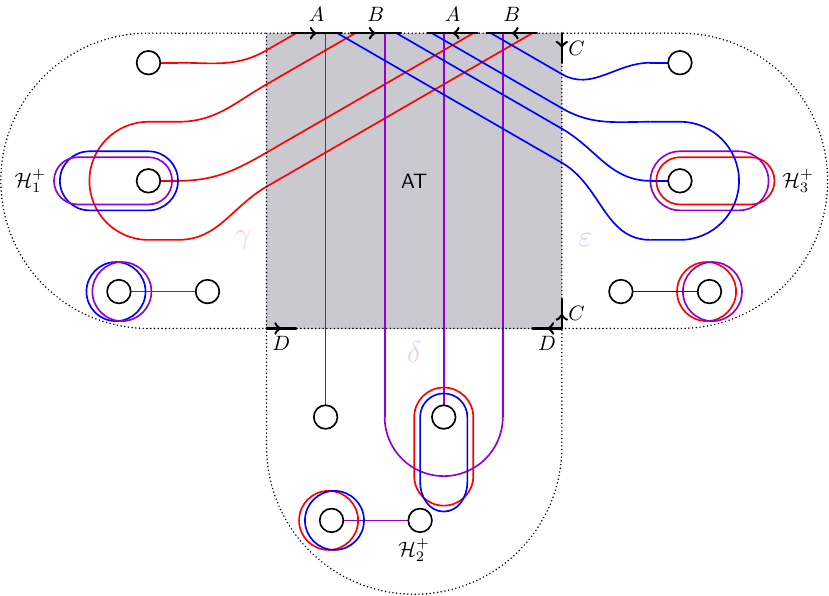}
		\end{center}
		\caption{An example of an $\AT_{1,2,3}$ obtained by gluing triples to $\AT(\mathcal{Z})$.}
	\end{figure}
	\begin{proposition}\label{FdeltadeltaProp}
		If $\mathcal{H}_2$ is admissible, the dg-bimodule homomorphism
		\begin{align}
			F_{\delta,\delta}:\mathcal{A}_{\bm{\gamma},\bm{\delta}}\boxtimes\widehat{\CFD}(\bm{\delta},\bm{\beta}_2^0)\otimes\overline{\widehat{\CFD}(\bm{\delta},\bm{\beta}_2^1)}\boxtimes\mathcal{A}_{\bm{\delta},\bm{\varepsilon}}\to\mathcal{A}_{\bm{\gamma},\bm{\varepsilon}}
		\end{align}
		defined by counting triangles in $\AT\cup_\delta\mathcal{H}_2^+$ with one corner at the bottom-graded generator of $\widehat{\CF}(\bm{\beta}_2^0,\bm{\beta}_2^1)$ is given up to homotopy by the map
		\begin{align}
			\rho\boxtimes\bm{u}^0\otimes\overline{\bm{v}}^1\boxtimes\sigma\mapsto\rho\overline{\bm{v}}^1(\bm{u}^1)\sigma,
		\end{align}
		where we regard $\overline{\bm{v}}^1$ as a map from $\widehat{\CFD}(\bm{\delta},\bm{\beta}_2^1)$ to the ring of idempotents $\mathcal{I}$ in $\mathcal{A}$.
	\end{proposition}
	\begin{proof}
		By definition, we have
		\begin{align}
			F_{\delta,\delta}(\rho\boxtimes\bm{u}^0\otimes\overline{\bm{v}}^1\boxtimes\sigma)=\sum_{\tau\in\mathfrak{S}(\AZ_{\gamma,\varepsilon})}\sum_{\ind(C)=0}\#\mathcal{M}^C(\rho\boxtimes\bm{u}^0,\overline{\bm{v}}^1\boxtimes\sigma,\tau\otimes\Theta^{\mathrm{bot}}_{\beta_2^0\beta_2^1})\tau,
		\end{align}
		where $C$ ranges over $\pi_2(\rho\boxtimes\bm{u}^0,\overline{\bm{v}}^1\boxtimes\sigma,\tau\otimes\Theta^{\mathrm{bot}}_{\beta_2^0\beta_2^1})$ and $\mathcal{M}^C(\rho\boxtimes\bm{u}^0,\overline{\bm{v}}^1\boxtimes\sigma,\tau\otimes\Theta^{\mathrm{bot}}_{\beta_2^0\beta_2^1})$ is the moduli space of pseudoholomorphic representatives of the class $C$. By the pairing theorem for triangles \cite[Proposition 5.35]{LOT2}, this map is homotopic to the one given by counting rigid triangles paired with sequences of bigons. Since there are no positive domains of rigid holomorphic triangles in $\AT$ which meet the boundary by Lemma \ref{TriangleLemma}, and because $\mathcal{H}_2^+$ is obtained by a small Hamiltonian translation, this tells us that $F_{\delta,\delta}$ is homotopic to the map
		\begin{align}
			\rho\boxtimes\bm{u}^0\otimes\overline{\bm{v}}^1\boxtimes\sigma\mapsto\sum_{\tau\in\mathfrak{S}(\AZ_{\gamma,\varepsilon})}\sum_{\ind(C)=0}\#\mathcal{M}^C_\times(\rho,\sigma,\tau,\bm{u}^0,\overline{\bm{v}}^1,\Theta^{\mathrm{bot}}_{\beta_2^0\beta_2^1})\tau,
		\end{align}
		where the moduli space $\mathcal{M}^C_\times(\rho,\sigma,\tau,\bm{u}^0,\overline{\bm{v}}^1,\Theta^{\mathrm{bot}}_{\beta_2^0\beta_2^1})$ is defined by
		\begin{align}
			\mathcal{M}^C_\times(\rho,\sigma,\tau,\bm{u}^0,\overline{\bm{v}}^1,\Theta^{\mathrm{bot}}_{\beta_2^0\beta_2^1})=\bigsqcup_{A+B=C}\mathcal{M}^A(\rho,\sigma,\tau)\times\mathcal{M}^B(\bm{u}^0,\overline{\bm{v}}^1,\Theta^{\mathrm{bot}}_{\beta_2^0\beta_2^1}),
		\end{align}
		where $A$ and $B$ are provincial domains in $\AT$ and $\mathcal{H}_2^+$, respectively. Here, $\mathcal{M}^A(\rho,\sigma,\tau)$ is the moduli space of rigid pseudoholomorphic triangles of class $A$ from $\rho\otimes\sigma$ to $\tau$ and $\mathcal{M}^B(\bm{u}^0,\bm{v}^1,\Theta^{\mathrm{bot}}_{\beta_2^0\beta_2^1})$ is the moduli space of rigid provincial triangles from $\bm{u}^0\otimes\overline{\bm{v}}^1$ to $\Theta^{\mathrm{bot}}_{\beta_2^0\beta_2^1}$ representing the class $B$. Note that this latter moduli space is empty unless $\bm{u}^0$ and $\bm{v}^1$ have the same left-idempotent $\iota^{01}$, which is then necessarily also the right-idempotent for $\rho$ and the left-idempotent for $\sigma$ in order for $\rho\boxtimes\bm{u}^0\otimes\overline{\bm{v}}^1\boxtimes\sigma$ to be nonzero. Together with additivity of the embedded index for disjoint unions and the fact that the index of a class with a pseudoholomorphic representative is non-negative, this then implies that
		\begin{align}
			\begin{split}
				&F_{\delta,\delta}(\rho\boxtimes\bm{u}^0\otimes\overline{\bm{v}}^1\boxtimes\sigma)\\&\simeq\sum_{\tau\in\mathfrak{S}(\AZ_{\gamma,\varepsilon})}\sum_{\ind(A)=\ind(B)=0}\#\mathcal{M}^A(\rho,\sigma,\tau)\#\mathcal{M}^B(\bm{u}^0,\overline{\bm{v}}^1,\Theta^{\mathrm{bot}}_{\beta_2^0\beta_2^1})\tau.
			\end{split}
		\end{align}
		However, this gives us
		\begin{align}
			\begin{split}
				&F_{\delta,\delta}(\rho\boxtimes\bm{u}^0\otimes\overline{\bm{v}}^1\boxtimes\sigma)\\&\simeq\left(\sum_{\ind(B)=0}\#\mathcal{M}^B(\bm{u}^0,\overline{\bm{v}}^1,\Theta^{\mathrm{bot}}_{\beta_2^0\beta_2^1})\right)\sum_{\tau\in\mathfrak{S}(\AZ_{\gamma,\varepsilon})}\sum_{\ind(A)=0}\#\mathcal{M}^A(\rho,\sigma,\tau)\tau,
			\end{split}
		\end{align}
		and the map
		\begin{align}
			\rho\otimes\sigma\mapsto\sum_{\tau\in\mathfrak{S}(\AZ_{\gamma,\varepsilon})}\sum_{\ind(A)=0}\#\mathcal{M}^A(\rho,\sigma,\tau)\tau
		\end{align}
		is precisely the multiplication map $\mathcal{A}\otimes\mathcal{A}\to\mathcal{A}$ by \cite[Proposition 4.8]{Auroux2010}. We then have
		\begin{align}
			F_{\delta,\delta}(\rho\boxtimes\bm{u}^0\otimes\overline{\bm{v}}^1\boxtimes\sigma)\simeq\rho\left(\sum_B\#\mathcal{M}^B(\bm{u}^0,\overline{\bm{v}}^1,\Theta^{\mathrm{bot}}_{\beta_2^0\beta_2^1})\iota^0\iota^1\right)\sigma,
		\end{align}
		where $\iota^0$ is the left-idempotent for $\bm{u}^0$ and $\iota^1$ is the right-idempotent for $\overline{\bm{v}}^1$, which we may insert at no cost since the space $\mathcal{M}^B(\bm{u}^0,\overline{\bm{v}}^1,\Theta^{\mathrm{bot}}_{\beta_2^0\beta_2^1})$ of provincial triangles is empty unless $\iota^0=\iota^1=\iota^{01}$, in which case we have $\rho\sigma=\rho\iota^{01}\sigma$. We claim that the map $L:\widehat{\CFD}(\bm{\delta},\bm{\beta}_2^0)\otimes\overline{\widehat{\CFD}(\bm{\delta},\bm{\beta}^1)}\to\mathcal{A}$ given by
		\begin{align}
			\bm{u}^0\otimes\overline{\bm{v}}^1\mapsto\sum_{\ind(B)=0}\#\mathcal{M}^B(\bm{u}^0,\bm{v}^1,\Theta^{\mathrm{bot}}_{\beta_2^0\beta_2^1})\iota^0\iota^1
		\end{align}
		is homotopic to the perturbed evaluation map $ev\circ(\Psi_{0\to 1}\otimes\id)$ given on generators by
		\begin{align}
			\bm{u}^0\otimes\overline{\bm{v}}^1\mapsto \overline{\bm{v}}^1(\bm{u}^1).
		\end{align}
		However, $L$ is dual to the type-$D$ morphism $R:\widehat{\CFD}(\bm{\delta},\bm{\beta}_2^0)\to\mathcal{A}\otimes\widehat{\CFD}(\bm{\delta},\bm{\beta}^1)$ given by
		\begin{align}
			\bm{u}^0\mapsto\sum_{\bm{v}^1\in\mathfrak{S}(\bm{\delta},\bm{\beta}_2^1)}\sum_{\ind(B)=0}\#\mathcal{M}^B(\bm{u}^0,\Theta^{\mathrm{top}}_{\beta_2^0\beta_2^1},\bm{v}^1)\iota^0\iota^1\otimes\bm{v}^1
		\end{align}
		which is filtered with respect to the the filtrations $\mathcal{F}$ and $\mathcal{F}^1$ defined in Lemma \ref{PerturbationLemma}. As a filtered map, this has filtration preserving part given by $\Psi_{0\to 1}$ since $\Psi_{0\to 1}$ is a summand of $R$ and $R$ is a summand of $\widehat{F}_{\delta\beta_2^0\beta_2^1}^{\mathrm{top}}$. This implies that $R$ is an isomorphism and the same neck-stretching argument used in \cite[Lemma 5.4]{Guth2022} to show that $\widehat{F}_{\delta\beta^0_2\beta^1_2}^{\mathrm{top}}$ is homotopic to $\Psi_{0\to 1}$ can be used to show that $R$ is homotopic to $\Psi_{0\to 1}$. Such a homotopy then induces a homotopy between the corresponding dual maps. Since the dual of $\Psi_{0\to 1}$ is $ev\circ(\Psi_{0\to 1}\otimes\id)$, this proves the desired result.
	\end{proof}
	\begin{figure}
		\begin{center}
			\includegraphics[scale=1]{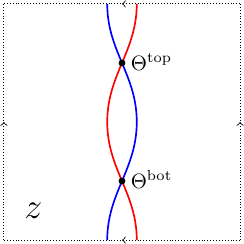}
		\end{center}
		\caption{A standard genus 1 Heegaard diagram for $S^2\times S^1$ with top- and bottom-graded generators labeled.}
		\label{standard}
	\end{figure}
	\begin{theorem}\label{AT-Comp}
		Let
		\begin{align}
			\widehat{G}_{\AT}:\Mor^{\mathcal{A}}(Y_1,Y_2)\otimes_\F\Mor^{\mathcal{A}}(Y_2,Y_3)\to\Mor^{\mathcal{A}}(Y_1,Y_3)
		\end{align}
		be the composite
		\begin{align}
			\begin{tikzcd}[ampersand replacement=\&]
				\Mor^{\mathcal{A}}(Y_1,Y_2)\otimes_\F\Mor^{\mathcal{A}}(Y_2,Y_3)\arrow[r,"\widehat{G}_{\AT}"]\arrow[d,"\cong"'] \& \Mor^{\mathcal{A}}(Y_1,Y_3)\\
				\widehat{\CF}(\mathcal{H}_1\cup\mathcal{H}_2)\otimes_{\F}\widehat{\CF}(\mathcal{H}_2\cup\mathcal{H}_3) \& \widehat{\CF}(\mathcal{H}_1\cup\mathcal{H}_3)\arrow[u,"\cong"']\\
				(\widehat{\CF}(\mathcal{H}_1\cup\mathcal{H}_2)\otimes V^{\otimes g_3})\otimes_{\F}(\widehat{\CF}(\mathcal{H}_2\cup\mathcal{H}_3)\otimes V^{\otimes g_1})\arrow[u,"\textup{1-handle}",hookleftarrow]\arrow[r,"\widehat{F}_{\AT_{1,2,3}}"'] \& \widehat{\CF}(\mathcal{H}_1\cup\mathcal{H}_3)\otimes V^{\otimes g_2}\arrow[u,"\textup{3-handle}"',twoheadrightarrow]
			\end{tikzcd}
		\end{align}
		where we take the model $\overline{\widehat{\CFD}(\mathcal{H}_i)}\boxtimes\mathcal{A}\boxtimes\widehat{\CFD}(\mathcal{H}_j)$ for $\widehat{\CF}(\mathcal{H}_i\cup\mathcal{H}_j)$, the vertical isomorphisms are the ones described above, $V
		$ is the two-dimensional model for $\widehat{\CF}(S^2\times S^1)$ given by the standard genus 1 Heegaard diagram for $S^2\times S^1$, $\widehat{F}_{\AT_{1,2,3}}$ is the map determined by the Heegaard triple $\AT_{1,2,3}$, and
		\begin{align*}
			\widehat{\CF}(Y)\stackrel{\textup{1-handle}}{\hookrightarrow}\widehat{\CF}(Y)\otimes V^{\otimes m}\cong\widehat{\CF}(Y\#(S^2\times S^1)^{\#m})
		\end{align*}
		and
		\begin{align*}
			\widehat{\CF}(Y)\otimes V^{\otimes n}\cong\widehat{\CF}(Y\#(S^2\times S^1)^{\#n})\stackrel{\textup{3-handle}}{\twoheadrightarrow} \widehat{\CF}(Y)
		\end{align*}
		are the usual 1-handle and 3-handle maps defined on generators by
		\begin{align*}
			\bm{x}\mapsto\bm{x}\otimes\Theta^{\mathrm{top}}
		\end{align*}
		and
		\begin{align}
			\bm{y}\otimes\theta\mapsto\begin{cases}
				\bm{y}\hspace{0.25cm}\textup{if $\theta=\Theta^{\mathrm{bot}}$}\\
				0\hspace{0.25cm}\textup{else,}
			\end{cases}
		\end{align}
		respectively, where $\Theta^{\mathrm{bot}}$ is the bottom-graded generator. Then $\widehat{G}_{\AT}$ agrees up to homotopy with the composition map $f\otimes g\mapsto g\circ f$.
	\end{theorem}
	\begin{proof}
		We assume that each of the bordered Heegaard triples $\mathcal{H}_i^+=(\overline{\Sigma}_i,\bm{\eta},\bm{\beta}_i^0,\bm{\beta}_i^1)$ are obtained by suitable small Hamiltonian perturbations so that Lemma \ref{PerturbationLemma} applies. By construction and the pairing theorem for triangles \cite[Proposition 3.35]{LOT2}, we have a decomposition $\widehat{G}_{\AT}\simeq\widehat{F}_{\gamma\beta_1^1\beta_1^0}^{\mathrm{top}}\boxtimes F_{\delta,\delta}\boxtimes\widehat{F}_{\varepsilon\beta_3^0\beta_3^1}^{\mathrm{top}}$ under the identifications $\Mor^{\mathcal{A}}(Y_i,Y_j)\cong\overline{\widehat{\CFD}(Y_i)}\boxtimes\mathcal{A}\boxtimes\widehat{\CFD}(Y_j)$.	Since the maps $\widehat{F}_{\gamma\beta_1^1\beta_1^0}^{\mathrm{top}}$ and $\widehat{F}_{\varepsilon\beta_3^0\beta_3^1}^{\mathrm{top}}$ are homotopic to the corresponding nearest point maps, Proposition \ref{FdeltadeltaProp} then tells us that $\widehat{G}_{\AT}$ is homotopic to the map given on basic morphisms by
		\begin{align}
			(\overline{\bm{t}}^1\boxtimes\rho\boxtimes\bm{u}^0)\otimes(\overline{\bm{v}}^1\boxtimes\sigma\boxtimes\bm{w}^0)\mapsto\overline{\bm{t}}^0\boxtimes\rho\overline{\bm{v}}^1(\bm{u}^1)\sigma\boxtimes\bm{w}^1,
		\end{align}
		which is precisely the composition map.
	\end{proof}
	\begin{corollary}
		Suppose that $\mathcal{H}_1$ and $\mathcal{H}_1'$ are bordered Heegaard diagrams for a bordered 3-manifold $Y_1$ differing by a single bordered Heegaard move, then the square
		\begin{align}
			\begin{tikzcd}[ampersand replacement=\&,column sep=1.5cm]
				\Mor^{\mathcal{A}}(\mathcal{H}_1,\mathcal{H}_2)\otimes\Mor^{\mathcal{A}}(\mathcal{H}_2,\mathcal{H}_3)\arrow[d,"\simeq"']\arrow[r,"f\otimes g\mapsto g\circ f"] \& \Mor^{\mathcal{A}}(\mathcal{H}_1,\mathcal{H}_3)\arrow[d,"\simeq"]\\
				\Mor^{\mathcal{A}}(\mathcal{H}_1',\mathcal{H}_2)\otimes\Mor^{\mathcal{A}}(\mathcal{H}_2,\mathcal{H}_3)\arrow[r,"f\otimes g\mapsto g\circ f"] \& \Mor^{\mathcal{A}}(\mathcal{H}_1',\mathcal{H}_3)
			\end{tikzcd}
		\end{align}
		commutes up to homotopy, where the vertical maps are given by the homotopy equivalences $\Mor^{\mathcal{A}}(\mathcal{H}_1,\mathcal{H}_i)\to\Mor^{\mathcal{A}}(\mathcal{H}_1',\mathcal{H}_i)$ induced by the Heegaard move. The analogous statement also holds for $\mathcal{H}_2$ and $\mathcal{H}_3$.
	\end{corollary}
	\begin{proof}
		In the case of finger moves and handleslides, this follows from Theorem \ref{AT-Comp} by associativity of triangle counts. In the case of stabilizations, up to some number of finger moves and handleslides, one may assume that the stabilization is performed in a neighborhood of the basepoint, in which case the vertical maps are isomorphisms.
	\end{proof}
	\section{4-manifolds with corners from bordered Heegaard triples}
	Just as one may represent a 4-manifold with boundary by a closed Heegaard triple and bordered 3-manifolds may be represented using (arced) bordered Heegaard diagrams \cite{Lipshitz2018}, we may describe 4-manifolds with boundary and corners using a suitable amalgamation of the two notions.
	\begin{definition}
		A genus $g$ \emph{arced bordered Heegaard triple} with $B$ boundary components is a quintuple $\mathcal{H}=(\overline{\Sigma},\bm{\gamma},\bm{\delta},\bm{\varepsilon},\bm{z})$, where:
		\begin{itemize}
			\item  $\overline{\Sigma}$ is a compact connected surface of genus $g$ with boundary components $\partial_1\overline{\Sigma},\dots,\partial_B\overline{\Sigma}$
			\item each $\bm{\eta}\in\{\bm{\alpha},\bm{\beta},\bm{\gamma}\}$ is a pairwise disjoint collection
			\begin{align*}
				\bm{\eta}=\{\eta_1^c,\dots,\eta_{g-T_\eta}\}\cup\bigcup\limits_{i=1}^B\{\eta_1^i,\dots,\eta_{2t_i^\eta}^i\},
			\end{align*}
			where $T_\eta=\sum\limits_{i=1}^Bt_i^\eta$, consisting of embedded arcs $\eta_j^i$ in $\overline{\Sigma}$ with boundary on $\partial_i\Sigma$ and circles $\eta^c_k$ in the interior of $\overline{\Sigma}$. We further impose the condition that if $t_i^\eta\neq 0$, then $t_i^\theta=0$ for $\bm{\theta}\neq\bm{\eta}$. In other words, this condition says that no two collections of curves meet the same boundary component nontrivially. For the sake of convenience, we denote the collection $\{\eta_1^c,\dots,\eta_{g-T_\eta}^c\}$ by $\bm{\eta}^c$ and the collections $\{\eta_1^i,\dots,\eta_{2t_i^\eta}^i\}$ by $\bm{\eta}^i$.
			\item $\bm{z}=(z;s_1,\dots,s_b)$ consists of an interior point $z\in\overline{\Sigma}$ disjoint from $\bm{\gamma}\cup\bm{\delta}\cup\bm{\varepsilon}$ together with embedded arcs $s_i$ in $\overline{\Sigma}\smallsetminus(\bm{\gamma}\cup\bm{\delta}\cup\bm{\varepsilon})$ connecting $z$ and $\partial_i\overline{\Sigma}$.
		\end{itemize}
		We also require that each of $\overline{\Sigma}\smallsetminus\bm{\gamma}$, $\overline{\Sigma}\smallsetminus\bm{\delta}$, and $\overline{\Sigma}\smallsetminus\bm{\varepsilon}$ is connected and that the collections $\bm{\gamma}$, $\bm{\delta},$ and $\bm{\varepsilon}$ intersect pairwise transversely. Lastly, we require that each component of $\partial\overline{\Sigma}$ is met by some $\bm{\eta}$. If $\bm{\eta}^i$ is the collection of arcs meeting $\partial_i\Sigma$ nontrivially, we will denote the induced (as in Lemma 4.4 of \cite{Lipshitz2018}) pointed matched circle by $\mathcal{Z}_i(\mathcal{H})$ or simply by $\mathcal{Z}_i$ when there is no risk of ambiguity.
	\end{definition}
	\begin{figure}
		\begin{center}
			\includegraphics[scale=1]{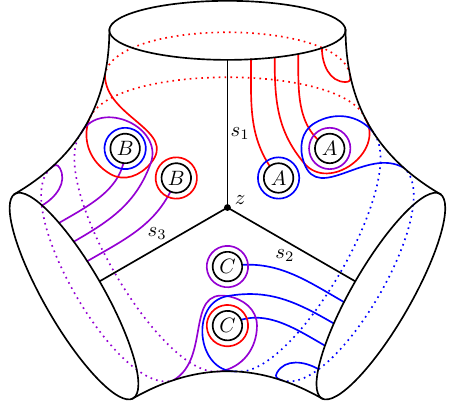}
		\end{center}
		\caption{A genus 3 bordered Heegaard triple $\mathcal{H}$ with three boundary components.}
		\label{drill1}
	\end{figure}
	Note that, for any two distinct collections $\bm{\eta},\bm{\theta}\in\{\bm{\alpha},\bm{\beta},\bm{\gamma}\}$, forgetting the third collection, filling in the now-empty boundary components with disks, and forgetting the arcs $s_{i_1},\dots,s_{i_f}$ which meet the filled boundary components, yields an arced bordered Heegaard diagram $\mathcal{H}^{\eta,\theta}=(\overline{\Sigma}_{\eta,\theta},\bm{\eta},\bm{\theta},\bm{z}_{\eta,\theta})$. Such a diagram determines a (strongly bordered) 3-manifold $Y_{\eta\theta}=Y(\mathcal{H}^{\eta,\theta})$ with $B-f$ boundary components by attaching 2-handles to $\overline{\Sigma}_{\eta,\theta}\times[0,1]$, analogous to \cite[Constructions 5.3 and 5.6]{Lipshitz2015}. From an arced bordered Heegaard triple $\mathcal{H}$, we will define a 4-manifold $X(\mathcal{H})$ with connected boundary and corners.
	\begin{remark}
		One could more generally allow bordered Heegaard triples $\mathcal{H}$ whose arcs connect multiple boundary components, in which case $\partial X(\mathcal{H})$ is a bordered sutured 3-manifold with corners following constructions analogous to those given by Zarev in \cite{Zarev2011}. However, we will not explore this construction here; we content ourselves to only consider the case $B\leq 3$.
	\end{remark}
	
	In addition to $Y_{\eta\theta}$, the arced bordered Heegaard diagram $\mathcal{H}^{\eta,\theta}$ specifies preferred disks $\Delta_j\subset\partial_jY_{\eta\theta}$, which are obtained as the images in $Y_{\eta\theta}$ of the ``faces'' of the 2-handles attached in the last step of the above construction, points $z_j\in\partial\Delta_j$ coming from the endpoints of $\bm{z}_{\eta,\theta}$, and homeomorphisms of triples $\phi_i:(F(\mathcal{Z}_j),D_j,z_j)\to(\partial_jY_{\eta\theta},\Delta_j,z_j)$ for each $j\neq i_1,\dots,i_f$, and an isotopy class $\nu_{\eta,\theta}$ of nowhere vanishing normal vector fields to $\bm{z}_{\eta,\theta}$ pointing into $\Delta_j$ at $z_j$. The data $(Y_{\eta\theta},\bm{\phi}_{\eta,\theta},\nu_{\eta,\theta})$, where $\bm{\phi}_{\eta,\theta}=\{\phi_j\}$ (note that this collection includes the data of the preferred disks and basepoints), is called the \emph{strongly bordered} 3-manifold associated to $\mathcal{H}^{\eta,\theta}$. We will abbreviate this data as $Y_{\eta\theta}$.
	
	\begin{construction}\label{CobConstruction}
		Let $\mathcal{H}=(\overline{\Sigma},\bm{\gamma},\bm{\delta},\bm{\varepsilon},\bm{z})$ be an an (arced) bordered Heegaard triple. For $\bm{\eta}\in\{\bm{\gamma},\bm{\delta},\bm{\varepsilon}\}$ meeting the boundary, construct a \emph{cornered handlebody} $\overline{U}_\eta$ as follows. Let $\overline{U}_0=\overline{\Sigma}\times[0,1]$ and let $\mathring{F}_\eta=F(\mathcal{Z}_\eta)\smallsetminus\mathrm{int}(D^2_\eta)$, where $D^2_\eta$ is the disk with $\partial D_\eta^2=Z_\eta$ used to construct $F(\mathcal{Z}_\eta)$ from the pointed matched circle $\mathcal{Z}_\eta=(Z_\eta,\bm{a}_\eta,M_\eta)$. Choose a closed collar neighborhood $[-\varepsilon,0]\times Z_\eta$ of $Z_\eta\subset\overline{\Sigma}$ such that $\{0\}\times Z_\eta$ is identified with $Z_\eta$ as in the following schematic figure.
		\begin{center}
			\includegraphics[scale=1]{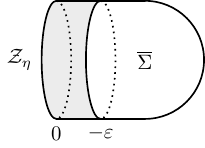}
		\end{center}
		Next, choose a closed tubular neighborhood $Z_\eta\times[0,1]$ of $Z_\eta$ in $ \mathring{F}_\eta$ and glue $\overline{U}_0$ to $[-\varepsilon,0]\times \mathring{F}_\eta$ by identifying the subsets $([-\varepsilon,0]\times Z_\eta)\times[0,1]\subset\overline{\Sigma}\times[0,1]$ and $[-\varepsilon,0]\times(Z_\eta\times[0,1])\subset[-\varepsilon,0]\times F_\eta$ as in
		\begin{center}
			\includegraphics[scale=1]{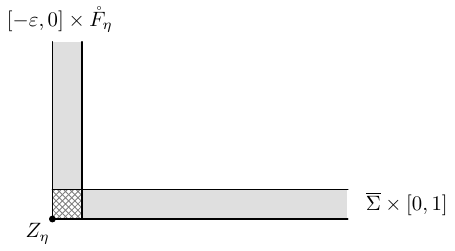}
		\end{center}
		and, similarly, attach a copy of $[-\varepsilon,0]\times D^2$ at each boundary component not met by $\bm{\eta}$ to obtain a new cornered 3-manifold $\overline{U}_1$ with two cornered boundary components, both of which are of the form $\Sigma_\eta:=\mathring{F}_\eta\cup_{\eta}\overline{\Sigma}\cup_\partial D^2\cup_\partial\stackrel{B-1}{\cdots}\cup_\partial D^2$, where $B=\#\pi_0(\partial\overline{\Sigma})$ and each surface in this union is glued to $\overline{\Sigma}$ at a 90 degree angle. For $\bm{\eta}$ not meeting any boundary component, instead attach a copy of $[-\varepsilon,0]\times D^2$ in this manner at each boundary component --- in this case, the resulting cornered 3-manifold has boundary components of the form $\Sigma_\varnothing:=\overline{\Sigma}\cup_\partial D^2\cup_\partial\stackrel{B}{\cdots}\cup_\partial D^2$. Now attach 3-dimensional 2-handles to the $\eta$-circles $\eta_i^c\times\{0\}\subset\overline{\Sigma}\times[0,1]$ as in
		\begin{center}
			\includegraphics[scale=1]{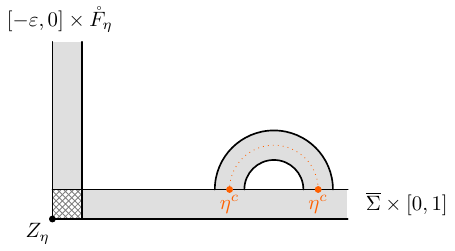}
		\end{center}
		to obtain a new 3-manifold $\overline{U}_2$ with two boundary components: a copy of $\Sigma_\eta$ or $\Sigma_\varnothing$ meeting $\overline{\Sigma}\times\{1\}$ and a genus $2k_\eta$ surface $S_\eta$, where $4k_\eta$ is the number of points in the boundary pointed matched circle corresponding to $\bm{\eta}$ (which is zero if $\bm{\eta}$ does not meet the boundary), which meets $\overline{\Sigma}\times\{0\}$.
		\begin{figure}
			\begin{center}
				\includegraphics[scale=1]{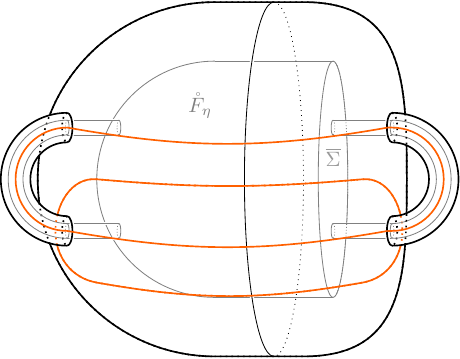}
			\end{center}
			\caption{A genus 1 example of a $\overline{U}_2$ in the case that $\bm{\eta}$ does meet the boundary.}
			\label{U1}
		\end{figure}
		\begin{figure}
			\begin{center}
				\includegraphics[scale=1]{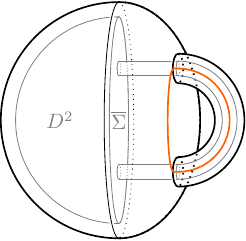}
			\end{center}
			\caption{A genus 1 example of a $\overline{U}_2$ in the case that $\bm{\eta}$ does not meet the boundary.}
			\label{U0}
		\end{figure}
		Next, if $\bm{\eta}$ meets the boundary, join each $\eta$-arc $\eta_i^a\times\{0\}\subset S_\eta$ to the core of the corresponding handle in $\{-\varepsilon\}\times\mathring{F}_\eta$ to obtain a collection of closed curves and attach a 3-dimensional 2-handle along each as in the following figure. If $\bm{\eta}$ does not meet the boundary, instead go on immediately to the next step.
		\begin{center}
			\includegraphics[scale=1]{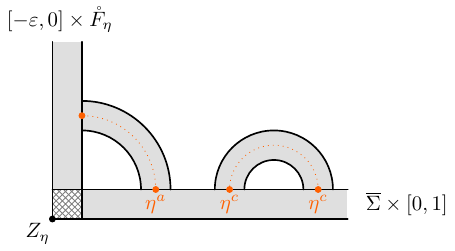}
		\end{center}
		This has the effect of replacing the boundary component $S_\eta$ with an $S^2$ boundary component. We then define $\overline{U}_\eta$ to be the result of filling this boundary component with a 3-ball as in
		\begin{center}
			\includegraphics[scale=1]{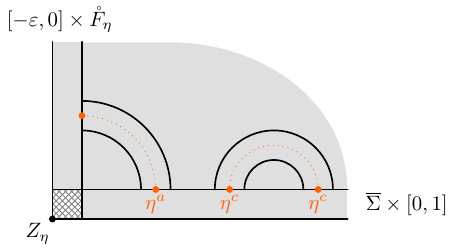}
		\end{center}
		--- the resulting space is a 3-manifold with boundary and corners, whose boundary stratum is $\partial_1\overline{U}_\eta=\Sigma_\eta$ or $\partial_1\overline{U}_\eta=\Sigma_\varnothing$, depending on whether or not $\bm{\eta}$ meets the boundary, and whose corner stratum is of the form $\partial_2\overline{U}_\eta=S^1\sqcup\stackrel{B}{\cdots}\sqcup S^1$. We then define a cornered 4-manifold $X(\mathcal{H})$ by
		\begin{align}
			X(\mathcal{H})=(\overline{\Sigma}\times\triangle)\cup_{\overline{\Sigma}\times e_\gamma}(\overline{U}_\gamma\times e_\gamma)\cup_{\overline{\Sigma}\times e_\delta}(\overline{U}_\delta\times e_\delta)\cup_{\overline{\Sigma}\times e_\varepsilon}(\overline{U}_\varepsilon\times e_\varepsilon),
		\end{align}
		where $\triangle$ is a triangle with edges labeled clockwise as $e_\gamma$, $e_\delta$, and $e_\varepsilon$, smoothing corners between the $\overline{U}_\eta$'s at the vertices $\overline{\Sigma}\times(e_\eta\cap e_\theta)$. Note that the boundary stratum $\partial_1 X(\mathcal{H})$ is connected and consists of the following two pieces. First, it contains each of the bordered 3-manifolds $Y_{\eta\theta}=Y(\mathcal{H}_{\eta\theta})$, where the diagrams $\mathcal{H}_{\eta\theta}=(\overline{\Sigma},\bm{\eta},\bm{\theta},z)$ are the bordered Heegaard diagrams obtained from $\mathcal{H}$ by deleting one of the collections of curves and filling the corresponding boundary component with a disk. Second, if $\bm{\theta}_1$ and $\bm{\theta}_2$ are the collections of curves not meeting the $\eta$-boundary, it contains a copy of
		\begin{align}
			\mathrm{Facet}_\eta:=S^1\times\triangle\cup_{S^1\times e_\eta}(\mathring{F}_\eta\times e_\eta)\cup_{S^1\times e_{\theta_1}}(D^2\times e_{\theta_1})\cup_{S^1\times e_{\theta_2}}(D^2\times e_{\theta_2}),
		\end{align}
		and there is one such ``facet'' for each $\bm{\eta}$ meeting $\partial\overline{\Sigma}$. These two distinguished parts of the boundary stratum meet in two copies of $F(\mathcal{Z}_\eta)$ and one copy of $S^2$. The union of these surfaces over all $\bm{\eta}$ meeting $\partial\overline{\Sigma}$ forms the corner stratum $\partial_2 X(\mathcal{H})$.
	\end{construction}
	In the single boundary component case, one may think of $X(\mathcal{H})$ schematically as in the following figure, which represents the $\delta$-bordered case.
	\begin{center}
		\includegraphics[scale=1]{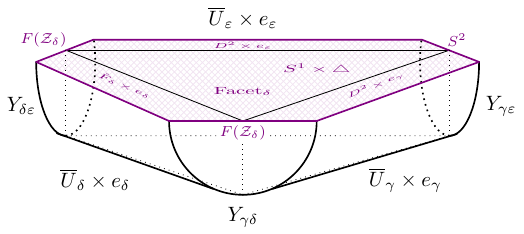}
	\end{center}
	However, this representation of $X(\mathcal{H})$ may be somewhat misleading: topologically, the space $\mathrm{Facet}_\eta$ is a closed 3-dimensional regular neighborhood of the singular surface $\mathring{F}_\eta\cup_\partial D^2\cup_\partial D^2$
	\begin{center}
		\includegraphics[scale=1]{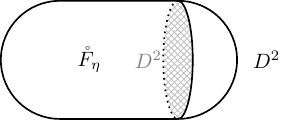}
	\end{center}
	--- i.e. $\mathrm{Facet}_\eta$ is a 3-dimensional pair of pants cobordism $-F(\mathcal{Z}_\eta)\sqcup F(\mathcal{Z}_\eta)\to S^2$. To see this, note that $\mathrm{Facet}_\eta$ is the result of gluing $\mathring{F}_\eta\times[0,1]$ and two copies of $D^2\times[0,1]$ to $S^1\times\triangle$ by identifying each of $\partial\mathring{F}_\eta\times[0,1]$ and the two copies of $\partial D^2\times[0,1]$ with one of $S^1\times e_\gamma$, $S^1\times e_\delta$, and $S^1\times e_\varepsilon$ so that $\partial\mathring{F}_\eta\times\{\frac{1}{2}\}$ and the two copies of $\partial D^2\times\{\frac{1}{2}\}$ are identified with the circles $S^1\times\{\mathrm{midpoint}\}$ depicted in the following figure and smoothing corners.
	\begin{center}
		\includegraphics[scale=1]{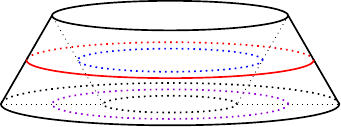}
	\end{center}
	\begin{remark}
		More generally, an arced bordered Heegaard $n$-tuple
		\begin{align}
			\mathcal{H}=(\overline{\Sigma}_g,\bm{\eta}_0,\dots,\bm{\eta}_{n-1},\bm{z})
		\end{align}
		with $B$ boundary components determines a cornered 4-manifold $X(\mathcal{H})$ whose boundary stratum consists of the bordered 3-manifolds $Y_{\eta_i\eta_{i+1}}$, with indices taken modulo $n$, together with facets $\mathrm{Facet}_{\eta_i}$ for each $i$ for which $\bm{\eta}_i$ intersects $\partial\overline{\Sigma}_g$ nontrivially. The constructions of $X(\mathcal{H})$ and the facets $\mathrm{Facet}_{\eta_i}$ are identical to the $n=3$ case except that we replace the triangle $\triangle$ with a planar $n$-gon.
	\end{remark}
	\subsection{Gluing}
	Let $\mathcal{H}=(\overline{\Sigma}_g,\bm{\gamma},\bm{\delta},\bm{\varepsilon},\bm{z})$ be an arced bordered Heegaard triple with three boundary components and let $\mathcal{H}_1=(\overline{\Sigma}_{g_1},\bm{\gamma}_1,\bm{\delta}_1,z_1)$, $\mathcal{H}_2=(\overline{\Sigma}_{g_2},\bm{\delta}_2,\bm{\varepsilon}_2,z_2)$, and $\mathcal{H}_3=(\overline{\Sigma}_{g_3},\bm{\varepsilon}_3,\bm{\gamma}_3,z_3)$ be $\gamma$-, $\delta$-, and $\varepsilon$-bordered Heegaard diagrams, respectively. Let $\mathcal{H}_{1,2,3}=\mathcal{H}\cup_\gamma\mathcal{H}_1^+\cup_\delta\mathcal{H}_2^+\cup_\varepsilon\mathcal{H}_3^+$ be the ordinary Heegaard triple that results from doubling the collections of curves in the $\mathcal{H}_i$ not meeting the boundary, labeling the new circles according to whichever label does not appear in $\mathcal{H}_i$, and gluing them to the corresponding boundary components of $\mathcal{H}$, as we did in the construction of $\AT_{1,2,3}$.
	\begin{proposition}\label{Gluing}
		If $\mathcal{H}_1$ and $\mathcal{H}_2$ are bordered Heegaard triples sharing a common boundary matching and $\mathcal{H}_2$ has one boundary component, then there is a diffeomorphism $X(\mathcal{H}_1\cup_{\partial}\mathcal{H}_2)\cong X(\mathcal{H}_1)\cup_{\mathrm{Facet}}X(\mathcal{H}_2)$, where $\mathrm{Facet}$ is the corresponding boundary facet. In particular, the 4-manifold $X(\mathcal{H}_{1,2,3})$ is diffeomorphic to
		\begin{align}
			X(\mathcal{H})\cup_{\mathrm{Facet}_\gamma}X(\mathcal{H}_1^+)\cup_{\mathrm{Facet}_\delta}X(\mathcal{H}_2^+)\cup_{\mathrm{Facet}_\varepsilon}X(\mathcal{H}_3^+).
		\end{align}
	\end{proposition}
	\begin{proof}
		Suppose that $\mathcal{H}_2$ is an $\eta$-bordered Heegaard diagram. The effect of gluing $\mathcal{H}_2$ to the $\eta$-boundary of $\mathcal{H}_1$ is as follows. First, the underlying surface $\overline{\Sigma}_g$ is replaced by $\overline{\Sigma}_g\cup_\eta\overline{\Sigma}_{g_2}$ which has the effect of gluing $\overline{\Sigma}_g\times\triangle$ to $\overline{\Sigma}_{g_2}\times\triangle$ in the obvious manner. Second, gluing the $\eta$-arcs which meet the boundary along their common endpoints corresponds to gluing the 3-dimensional 2-handles along the corresponding cores of the 1-handles in $\mathring{F}_\eta$ determined by the arcs. This has the effect of gluing the respective $\eta$-handlebodies along their $\mathring{F}_\eta$-boundaries. Lastly, for $\theta\neq\eta$, the respective $\theta$-handlebodies are glued along their disk boundaries. It is straightforward to see that these glued handlebodies are precisely the handlebodies obtained from the above construction using the glued diagram so this proves the result.
	\end{proof}
	\begin{figure}
		\begin{center}
			\includegraphics[scale=1.5]{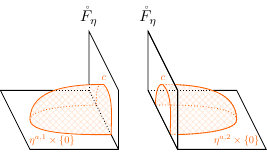}
			\includegraphics[scale=1.5]{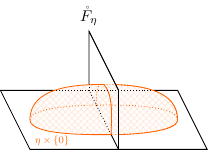}
		\end{center}
		\caption{The effect of gluing bordered Heegaard triples on the 2-handles attached to matched pairs of curves of the form $\eta^a\cup_\partial c$, where $c$ is the core of a 1-handle in $\mathring{F}_\eta$.}
		\label{ArcGluing}
	\end{figure}
	\begin{corollary}\label{PantsCorollary}
		The 4-manifold $X(\AT_{1,2,3})$ is diffeomorphic to the composition $W_2^{13,g_2}\circ W\circ(W_{-2}^{12,g_3}\sqcup W_{-2}^{23,g_1})$ of the pair of pants cobordism $W:Y_{12}\sqcup Y_{23}\to Y_{13}$ with the cobordisms $W_2^{ij,g_k}:Y_{ij}\to Y_{ij}\#(S^2\times S^1)^{g_k}$ obtained by surgery on 0-framed $g_k$-component unlinks in $Y_{ij}$ and their reverses $W_{-2}^{ij,g_k}:Y_{ij}\#(S^2\times S^1)^{g_k}\to Y_{ij}$. Thus, if $W_1^{ij,g}:Y_{ij}\to Y_{ij}\#(S^2\times S^1)^{\# g}$ and $W_3^{ij,g}:Y_{ij}\#(S^2\times S^1)^{\# g}\to Y_{ij}$ are the usual 1-handle and 3-handle cobordisms, then $W_3^{13,g_2}\circ X(\AT_{1,2,3})\circ(W_1^{12,g_3}\sqcup W_1^{23,g_1})$ is diffeomorphic to $W$.
	\end{corollary}
	\begin{proof}
		Suppose that $\overline{\mathcal{H}}=(\overline{\Sigma}_g,\bm{\alpha},\bm{\beta},z)$ is an $\alpha$-bordered Heegaard triple and let $Y=Y(\overline{\mathcal{H}})$ be the corresponding bordered 3-manifold. We claim that the cornered 4-manifold $X=X(\overline{\mathcal{H}}^+)$ determined by the triple $\mathcal{H}^+=(\overline{\Sigma},\bm{\alpha},\bm{\beta}^0,\bm{\beta}^1,z)$ obtained by doubling $\bm{\beta}$ is diffeomorphic to the cobordism of pairs
		\begin{align}
			(-Y\sqcup Y,-\partial Y\sqcup\partial Y)\to((S^2\times S^1)^{\# g}\setminus B^3,S^2)
		\end{align}
		given by the complement of a regular neighborhood of the cornered handlebody $\overline{U}_\beta\times\{0\}$ in $Y\times[-1,1]$. To see this, recall from \cite[Proposition 4.3]{Ozsvath2006} that if $\mathcal{H}'=(\Sigma_0,\bm{\alpha}',\bm{\beta}',z)$ is any Heegaard diagram for a closed 3-manifold $Y'$ and $(\Sigma_0,\bm{\alpha}',\bm{\beta}',\bm{\gamma},z)$ is such that $\bm{\gamma}$ is obtained by a small Hamiltonian translation of $\bm{\beta}'$, then the 4-manifold $X_{\alpha'\beta'\gamma}$ determined by this diagram is diffeomorphic to $Y'\times[-1,1]$ with a regular neighborhood of the handlebody $U_{\beta'}\times\{0\}$ deleted, i.e. the cobordism obtained by attaching 2-handles to a 0-framed unlink in a Euclidean ball in $Y'$. In particular, this is the case if $\mathcal{H}'=\mathcal{H}_0\cup_\partial\mathcal{H}$ for some other $\alpha$-bordered Heegaard diagram $\mathcal{H}_0$. The claim then follows from the previous proposition.
		
		The first statement now follows from Proposition \ref{Gluing} together with the observation that the surface underlying the triple $\AT$ is naturally identified with $F(\mathcal{Z})$ with three disks removed and the fact that deleting any pair of curves from $\AT$ determines a bordered Heegaard diagram for $F(\mathcal{Z})\times I$ after filling the now-empty boundary component with a disk. The second statement then follows from the fact that the 2- and 3-handles in $W_{3}^{13,g_2}\circ W_{2}^{13,g_2}$ and the 1- and 2-handles in $(W_{-2}^{12,g_3}\sqcup W_{-2}^{23,g_1})\circ(W_1^{12,g_3}\sqcup W_1^{23,g_1})$ cancel.
	\end{proof}
	Another way of thinking about these results is as follows. Given a closed 3-manifold $Y$, we have two distinct ways of decomposing $Y$ into 3-manifolds with boundary: we can either decompose $Y$ as $Y=U_\alpha\cup_\Sigma U_\beta$, where $U_\alpha$ and $U_\beta$ are handlebodies glued along a Heegaard surface $\Sigma$, or we can decompose it as $Y=-Y_1\cup_{F(\mathcal{Z})}Y_2$, where $Y_1$ and $Y_2$ are \emph{bordered} 3-manifolds which both have boundary parameterized by the same surface $F(\mathcal{Z})$.
	\begin{figure}
		\begin{center}
			\includegraphics[scale=1]{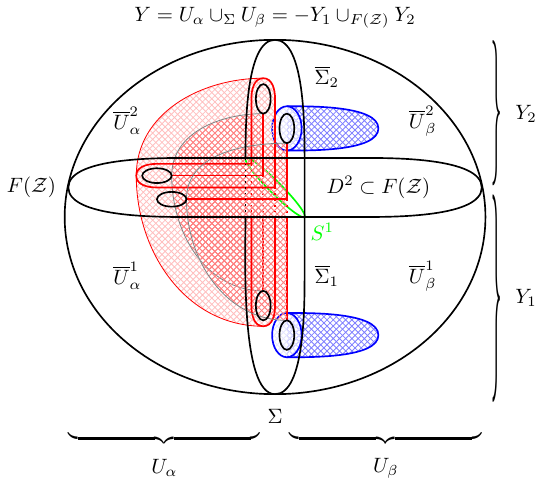}
		\end{center}
		\caption{Splitting a closed 3-manifold into two handlebodies along a Heegaard surface $\Sigma$ and into two bordered 3-manifolds along a surface $F(\mathcal{Z})$ transverse to the original. In each half-surface, the two small black circles are identified and, hence, such a pair represents a handle.}
		\label{Split}
	\end{figure}
	\begin{figure}
		\begin{center}
			\includegraphics[scale=1]{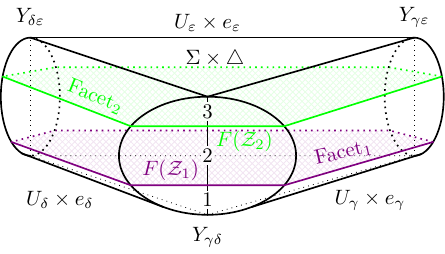}
		\end{center}
		\caption{Slicing a 4-manifold with boundary obtained from a closed Heegaard triple $\mathcal{H}=(\Sigma,\bm{\gamma},\bm{\delta},\bm{\varepsilon},\bm{z})$ along two facets. In this schematic example, the Heegaard surface $\Sigma$ decomposes as $\Sigma=\overline{\Sigma}_1\cup\overline{\Sigma}_2\cup\overline{\Sigma}_3$ so each of the 3-manifolds $Y_{\eta\theta}$ decomposes into bordered 3-manifolds as $Y_{\eta\theta}=Y_{\eta\theta}^1\cup_{F_1}Y_{\eta\theta}^2\cup_{F_2}Y_{\eta\theta}^3$ and each handlebody $U_\eta$ decomposes into cornered handlebodies as $U_\eta=\overline{U}_\eta^1\cup_{F_1\cap U_\eta}\overline{U}_\eta^2\cup_{F_2\cap U_\eta}\overline{U}_\eta^3$.}
	\end{figure}
	Here, we have chosen this second splitting to be one obtained by cutting a closed Heegaard diagram $(\Sigma,\bm{\alpha},\bm{\beta},z)$ for the decomposition $Y=U_\alpha\cup_\Sigma U_\beta$ along some circle which intersects one of the pairs of curves, giving us two bordered Heegaard diagrams with the same pointed matched circle $\mathcal{Z}$. In this case, the copy of the surface $F(\mathcal{Z})$ sitting inside of $Y$ meets $\Sigma$ transversely in a single separating copy of $S^1$. Therefore, each $Y_i$ decomposes as a union of two cornered handlebodies $Y_i=\overline{U}_\alpha^i\cup_{\Sigma\cap Y_i}\overline{U}_\beta^i$ and each handlebody $U_\eta$ decomposes similarly as $U_\eta=\overline{U}_\eta^1\cup_{F(\mathcal{Z})\cap U_\eta}\overline{U}_\eta^2$. This allows us to decompose $Y$ into four ``quadrants'' which are compatible with the (restrictions of) the gluings in both decompositions of $Y$ (cf. Figure \ref{Split}). These quadrants are precisely the cornered handlebodies from Construction \ref{CobConstruction}. If we had instead started with a closed Heegaard triple $\mathcal{H}=(\Sigma,\bm{\gamma},\bm{\delta},\bm{\varepsilon},\bm{z})$, separated $\Sigma$ along a circle intersecting exactly one of the sets of curves to obtain a decomposition $\Sigma=\overline{\Sigma}_1\cup_\partial\overline{\Sigma}_2$, and glued the cornered handlebodies meeting $\overline{\Sigma}_i$ to $\overline{\Sigma}_i\times\triangle$ to obtain $X(\mathcal{H}_i)$, then the complement of the bordered 3-manifolds $Y_{\eta\theta}$ in $\partial_1 X(\mathcal{H}_i)$ is precisely the interior of a facet so gluing $X(\mathcal{H}_1)$ and $X(\mathcal{H}_2)$ along their respective boundary facets yields the original 4-manifold $X(\mathcal{H})$.
	\section{The main theorem}
	In \cite{Zemke2021,Zemke2021b}, Zemke extends the minus and hat versions of Heegaard Floer homology to give monoidal functors out of the monoidal category of (multi)-pointed 3-manifolds and cobordisms between them equipped with embedded ribbon graphs connecting the basepoints. Given a closed Heegaard triple $(\Sigma,\bm{\gamma},\bm{\delta},\bm{\varepsilon},\bm{z})$, let $X_{\gamma\delta\varepsilon}$ be the smooth 4-manifold with boundary $\partial X_{\gamma\delta\varepsilon}=-Y_{\gamma\delta}\sqcup-Y_{\delta\varepsilon}\sqcup Y_{\gamma\varepsilon}$ defined by
	\begin{align}
		X_{\gamma\delta\varepsilon}=(\Sigma\times\triangle)\cup_{\Sigma\times e_\gamma}(U_\gamma\times e_\gamma)\cup_{\Sigma\times e_\delta}(U_\delta\times e_\delta)\cup_{\Sigma\times e_\varepsilon}(U_\varepsilon\times e_\varepsilon),
	\end{align}
	i.e. the pair of pants cobordism, as in \cite[Section 8]{OzsSz2004}. In \cite[Section 9]{Zemke2021}, Zemke endows $X_{\gamma\delta\varepsilon}$ with an embedded trivalent graph $\Gamma_{\gamma\delta\varepsilon}$ as follows: let $v_0\in\triangle$ be an interior point and define $\Gamma_0\subset\triangle$ by attaching a straight line segment extending radially from $v_0$ to each of the three vertices of the triangle. Then one defines $\Gamma_{\gamma\delta\varepsilon}:=\bm{z}\times\Gamma_0$ and gives this graph a ribbon structure by cyclically ordering the edges by endowing the ends of $X_{\gamma\delta\varepsilon}$ with the cyclic order $(-Y_{\gamma\delta},-Y_{\delta\varepsilon},Y_{\gamma\varepsilon})$.
	\begin{theorem}[{\cite[Theorem 9.1]{Zemke2021}}]
		Suppose that $(\Sigma,\bm{\gamma},\bm{\delta},\bm{\varepsilon},\bm{z})$ is a closed pointed Heegaard triple. Let
		\begin{align}
			(X_{\gamma\delta\varepsilon},\Gamma_{\gamma\delta\varepsilon}):(Y_{\gamma\delta}\sqcup Y_{\delta\varepsilon},\bm{z}\sqcup\bm{z})\to (Y_{\gamma\varepsilon},\bm{z})
		\end{align}
		be the ribbon graph cobordism described above. Then, if $\mathfrak{s}\in\mathit{Spin}^c(X_{\gamma\delta\varepsilon})$, the graph cobordism map
		\begin{align*}
			F^B_{X_{\gamma\delta\varepsilon},\Gamma_{\gamma\delta\varepsilon},\mathfrak{s}}:\CF^-(\Sigma,\bm{\gamma},\bm{\delta};\mathfrak{s}|_{Y_{\gamma\delta}})\otimes_{\F[U]} \CF^-(\Sigma,\bm{\delta},\bm{\varepsilon};\mathfrak{s}|_{Y_{\delta\varepsilon}})\to\CF^-(\Sigma,\bm{\gamma},\bm{\varepsilon};\mathfrak{s}|_{Y_{\gamma\varepsilon}})
		\end{align*}
		is chain homotopic to the holomorphic triangle map $F^-_{\alpha,\beta,\gamma,\mathfrak{s}}$.
	\end{theorem}
	\begin{corollary}\label{ZemkeCorollary}
		The hat Heegaard Floer analogue of \cite[Theorem 9.1]{Zemke2021} holds.
	\end{corollary}
	\begin{theorem}[{\cite[Theorem 1.2]{Zemke2021b}}]\label{ZemkeThm2}
		If $(W,\Gamma):(Y_0,\bm{p}_0)\to(Y_1,\bm{p}_1)$ is a graph cobordism, then the graph cobordism map $\widehat{F}_{W,\Gamma}:\widehat{\CF}(Y_0,\bm{p}_0)\to\widehat{\CF}(Y_1,\bm{p}_1)$ is functorial with respect to composition of cobordisms and if $\Gamma$ is a path connecting $\bm{p}_0$ to $\bm{p}_1$, then $\widehat{F}_{W,\Gamma}$ is homotopic to the cobordism map defined by Ozsv\'{a}th--Szab\'{o} in \textup{\cite{Ozsvath2006}}.
	\end{theorem}
	Note that the pair of pants cobordism with its embedded ribbon graph decomposes as $(X_{\gamma\delta\varepsilon},\Gamma_{\gamma\delta\varepsilon})=(W_1\cup_{Y_{12}\# Y_{23}}W_2,\Gamma_1\cup\Gamma_2)$, where $(W_1,\Gamma_1):Y_{12}\sqcup Y_{23}\to Y_{12}\# Y_{23}$ is the connected sum cobordism with an embedded trivalent graph $\Gamma_1$, and $(W_2,\Gamma_2):Y_{12}\# Y_{23}\to Y_{13}$ is the 2-handle cancellation cobordism equipped with an embedded path $\Gamma_2$ between basepoints. By \cite[Proposition 8.1]{Zemke2021}, the graph cobordism map $\widehat{F}_{W_1,\Gamma_1}:\widehat{\CF}(Y_{12})\otimes\widehat{\CF}(Y_{23})\to\widehat{\CF}(Y_{12}\# Y_{23})$ is homotopic to Ozsv\'{a}th--Szab\'{o}'s connected sum isomorphism. By the previous theorem, the map $\widehat{F}_{W_2,\Gamma_2}$ is homotopic to the map $\widehat{F}_{W_2}:\widehat{\CF}(Y_{12}\# Y_{23})\to\widehat{\CF}(Y_{13})$ defined by Ozsv\'{a}th--Szab\'{o} in \cite{Ozsvath2006}. With these facts in hand, we are now ready to prove Theorem \ref{MainTheorem}.
	\begin{theorem}[Theorem \ref{MainTheorem}]
		Let $Y_1$, $Y_2$, and $Y_3$ be bordered 3-manifolds, all of which have boundaries parameterized by the same surface $F$, and let $\mathcal{A}=\mathcal{A}(-F)$ be the algebra associated to $-F$. Let $Y_{ij}=-Y_i\cup_\partial Y_j$ and consider the pair of pants cobordism $W:Y_{12}\sqcup Y_{23}\to Y_{13}$. Then the composition map
		\begin{align}
			\Mor^{\mathcal{A}}(Y_1,Y_2)\otimes\Mor^{\mathcal{A}}(Y_2,Y_3)\to\Mor^{\mathcal{A}}(Y_1,Y_3)
		\end{align}
		given by $f\otimes g\mapsto g\circ f$ fits into a homotopy commutative square of the form
		\begin{align}
			\begin{tikzcd}[ampersand replacement=\&,column sep=1.5cm]
				\Mor^{\mathcal{A}}(Y_1,Y_2)\otimes\Mor^{\mathcal{A}}(Y_2,Y_3)\arrow[d,"\simeq"']\arrow[r,"f\otimes g\mapsto g\circ f"] \& \Mor^{\mathcal{A}}(Y_1,Y_3)\arrow[d,"\simeq"]\\
				\widehat{\CF}(Y_{12})\otimes\widehat{\CF}(Y_{23})\arrow[r,"\widehat{F}_W"]\&\widehat{\CF}(Y_{13})
			\end{tikzcd}
		\end{align}
		where $\widehat{F}_W$ is the map induced by $W$ and the vertical maps come from the pairing theorem of \textup{\cite{Lipshitz_2011}}.
	\end{theorem}
	\begin{proof}
		By Corollary \ref{ZemkeCorollary} and Theorem \ref{ZemkeThm2}, the maps $\widehat{G}_{\AT}$ and $\widehat{F}_W$ are homotopic. The result then follows from Theorem \ref{AT-Comp}.
	\end{proof}
	This immediately implies the following assertion of Lipshitz--Ozsv\'{a}th--Thurston in \cite[Section 1.5]{Lipshitz_2011}.
	\begin{corollary}
		The Yoneda composition map
		\begin{align}
			\Ext(Y_1,Y_2)\otimes_\F\Ext(Y_2,Y_3)\to\Ext(Y_1,Y_3),
		\end{align}
		where $\Ext(Y_i,Y_j):=\Ext(\widehat{\CFD}(Y_i),\widehat{\CFD}(Y_j))$, coincides with the map
		\begin{align}
			\widehat{\HF}(-Y_1\cup_\partial Y_2)\otimes_\F\widehat{\HF}(-Y_2\cup_\partial Y_3)\to\widehat{\HF}(-Y_1\cup_\partial Y_3)
		\end{align}
		induced by $W$.
	\end{corollary}
	\section{Application: an algorithm for computing $\widehat{F}_X$}
	As a consequence of Theorem \ref{MainTheorem}, we describe an algorithm for computing the morphism $\widehat{\HF}(Y_0)\to\widehat{\HF}(Y_1)$ associated to an arbitrary cobordism $X:Y_0\to Y_1$ between closed 3-manifolds. As in previous sections, we will abbreviate the notation for morphism spaces by omitting the symbols $\widehat{\CFD}$ and $\widehat{\CFDA}$: if $Y_0$ and $Y_1$ are 3-manifolds with boundary parametrized by $F(\mathcal{Z})$, then
	\begin{align}
		\Mor^{\mathcal{A}}(Y_0,Y_1):=\Mor^{\mathcal{A}(-\mathcal{Z})}(\widehat{\CFD}(Y_0),\widehat{\CFD}(Y_1))
	\end{align}
	and if $\varphi:F(\mathcal{Z})\to F(\mathcal{Z})$ is a diffeomorphism, then we define
	\begin{align}
		\Mor^{\mathcal{A}}(Y_0,\varphi\boxtimes Y_1):=\Mor^{\mathcal{A}(-\mathcal{Z})}(\widehat{\CFD}(Y_0),\widehat{\CFDA}(\varphi)\boxtimes\widehat{\CFD}(Y_1)),
	\end{align}
	where $\widehat{\CFDA}(\varphi)$ is the type-$\mathit{DA}$ bimodule of the mapping cylinder of $\varphi$. In \cite{Ozsvath2006}, Ozsv\'{a}th--Szab\'{o} define a map $\widehat{F}_X$ as follows: first decompose $X$ as $X=X_3\circ X_2\circ X_1$, where $X_1:Y_0\to Y_0'$ is a cobordism consisting entirely of 1-handles, $X_2:Y_0'\to Y_1'$ is a cobordism consisting of 2-handles, and $X_3:Y_1'\to Y_1$ is a cobordism consisting of 3-handles. They then define maps $\widehat{F}_{X_i}$, $i=1,2,3$, between the Floer complexes of the respective 3-manifolds associated to each type of handle, take $\widehat{F}_X=\widehat{F}_{X_3}\circ\widehat{F}_{X_2}\circ\widehat{F}_{X_1}$, and show that the resulting map on homology is well-defined and invariant under Kirby moves and, hence, is a 4-manifold invariant (see also \cite{Juhasz2021} and \cite{zemke2019graph}). The maps $\widehat{F}_{X_1}$ and $\widehat{F}_{X_3}$ are the same 1- and 3-handle maps described in Theorem \ref{AT-Comp}. We now describe the 2-handle map $\widehat{F}_{X_2}$. For notational simplicity, assume that $X$ is built entirely from 2-handles so that $\widehat{F}_{X}=\widehat{F}_{X_2}$. Then, $X$ is given by surgery on some framed link $\mathbb{L}\subset Y_0$. We recall the following definitions from \cite{Ozsvath2006}.
	\begin{definition}
		A \emph{bouquet} for $\mathbb{L}$ is an embedded 1-complex $B(\mathbb{L})\subset Y_0$ given by the union of $\mathbb{L}=K_1\cup\cdots\cup K_k$ with a collection of arcs connecting the link components $K_i$ to a fixed basepoint in $Y_0$.
	\end{definition}
	Fix a bouquet $B(\mathbb{L})$ for $\mathbb{L}$. Let $H_0$ be a regular neighborhood of $B(\mathbb{L})$, $F=\partial H_0$, and let $H_1=Y_0\setminus\mathrm{int}(H_0)$ be the complementary handlebody. Now define $H_0(\mathbb{L})$ to be the result of performing surgery on $\mathbb{L}\subset H_0$. Then $H_0(\mathbb{L})\cup_\partial H_1=Y_1$ and $H_0(\mathbb{L})\cup_\partial H_0\cong\#^{g(F)-k}S^2\times S^1$.
	\begin{definition}
		A Heegaard triple \emph{subordinate to the bouquet} $B(\mathbb{L})$ is a Heegaard triple $(\Sigma,\bm{\alpha},\bm{\beta},\bm{\gamma})$ such that
		\begin{enumerate}
			\item $(\Sigma,{\alpha_1,\dots,\alpha_g},{\beta_{k+1},\dots,\beta_g})$ is a diagram for the complement $H_1$ of the bouquet,
			\item $\gamma_{k+1},\dots,\gamma_g$ are small Hamiltonian translates of the $\beta_{k+1},\dots,\beta_g$,
			\item after surgering out the curves $\beta_{k+1},\dots,\beta_g$, the induced curves $\beta_i$ and $\gamma_i$, for $i=1,\dots,k$, lie in punctured tori $F_i\subset\partial H_1$ given by the boundaries of regular neighborhoods of the components $K_i$,
			\item the curves $\beta_i$, $i=1,\dots,k$, represent meridians of the $K_i$ and
			\begin{align*}
				\#(\beta_i\cap\gamma_j)=\begin{cases}
					0,\,\,\,i\neq j\\
					1,\,\,\,i=j
				\end{cases}
			\end{align*}
			with transverse intersection in the latter case,
			\item the curves $\gamma_i$, $i=1,\dots,k$, represent the framings of the components $K_i$ under the natural identification of $H_1(\partial\mathrm{nbd}(K_1\cup\cdots\cup K_k))$ with $H_1(\partial H_1)$.
		\end{enumerate}
	\end{definition}
	The 4-manifold $W$ specified by the triple $\mathcal{H}=(\Sigma,\bm{\alpha},\bm{\beta},\bm{\gamma})$ then has boundary components $-Y_0$, $\#^{g(F)-k}S^2\times S^1$, and $Y_1$ (cf. \cite[Proposition 4.3]{Ozsvath2006}) --- indeed $W$ is the pair of pants cobordism $Y_0\sqcup\#^{g(F)-k}S^2\times S^1\to Y_1$ --- and the map $\widehat{F}_X:\widehat{\CF}(Y_0)\to\widehat{\CF}(Y_1)$ is defined by taking $\widehat{F}_X(\bm{x})=\widehat{F}_{W}(\bm{x}\otimes\Theta^{\mathrm{top}})$, where the right-hand side is the holomorphic triangle counting map determined by $\mathcal{H}$, i.e. the pair of pants map for the handlebodies $H_0$, $H_0(\mathbb{L})$, and $H_1$. We may realize this construction in the morphism spaces formulation of Heegaard Floer homology as follows: suppose that $\theta^{\mathrm{top}}\in\Mor^{\mathcal{A}(-F)}(-H_0(\mathbb{L}),H_0)$ is a representative of the top-graded class in $\widehat{\HF}(\#^{g(F)-k}S^2\times S^1)$. Then, Theorem \ref{MainTheorem} tells us that there is a homotopy commutative square
	\begin{align}\label{square}
		\begin{tikzcd}[ampersand replacement=\&]
			\Mor^{\mathcal{A}(-F)}(-H_0,H_1)\arrow[r,"-\circ \theta^{\mathrm{top}}"]\&\Mor^{\mathcal{A}(-F)}(-H_0(\mathbb{L}),H_1)\arrow[d,"\simeq"]\\
			\widehat{\CF}(Y_0)\arrow[u,"\simeq"]\arrow[r,"\widehat{F}_X"]\&\widehat{\CF}(Y_1) 
		\end{tikzcd}
	\end{align}
	where the vertical arrows come from the pairing theorem of \cite{Lipshitz_2011}. An algorithm for computing $\widehat{\CFD}(H)$ for a handlebody $H$ was given by Lipshitz--Ozsv\'{a}th--Thurston in \cite{Lipshitz2014} (see also \cite{Zhan2016}).
	
	Now suppose that $X_1:Y_0\to Y_0'$ consists of a single 1-handle addition and let $\mathcal{A}_1=\mathcal{A}(-F)$. Then the map $\widehat{F}_{X_1}:\widehat{\CF}(Y_0)\to\widehat{\CF}(Y_0')$ can be computed by decomposing $Y_0'$ as $Y_0\#(S^2\times S^1)$, in which case $\widehat{F}_{X_1}(\bm{x})=\bm{x}\otimes\Theta^{\mathrm{top}}$. We now reinterpret this construction in the morphism spaces setting. If we take a Heegaard splitting $Y_0=H_0\cup_\varphi H_0$, where $H_0$ is a 0-framed handlebody of genus $g$ and $\varphi:\partial H_0\to\partial H_0$ is a diffeomorphism, then we automatically get a Heegaard splitting $Y_0'=H_0'\cup_{\varphi'}H_0'$, where $H_0'$ is the genus $g+1$ handlebody $H_0'=H_0\natural(D^2\times S^1)$ and $\varphi'=\varphi\#\mathrm{id}_{\mathbb{T}^2}$. This then gives us
	\begin{align}
		\widehat{\CF}(Y_0')&=\Mor^{\mathcal{A}_2}(-H_0',\varphi'\boxtimes H_0'),
	\end{align}
	where $\mathcal{A}_2=\mathcal{A}(F(-\mathcal{Z})\#\mathbb{T}^2)$, by the pairing theorem. If $\mathcal{H}_0$ is a bordered Heegaard diagram for $H_0$ and $\mathcal{H}_\varphi$ is an (arced) bordered Heegaard diagram for $\varphi$, we may obtain bordered Heegaard diagrams $\mathcal{H}_0'$ and $\mathcal{H}_{\varphi'}$ by appending a copy of the standard diagrams for $D^2\times S^1$ (with the 0-framing) and $\mathbb{T}^2\times[0,1]$ to $\mathcal{H}_0$ and $\mathcal{H}_\varphi$, respectively. This gives us isomorphisms\begin{align}
		\begin{split}
			\widehat{\CFD}(H_0')\cong\widehat{\CFD}(H_0)\otimes_\F\widehat{\CFD}(D^2\times S^1)
		\end{split}
	\end{align}
	and
	\begin{align}
		\begin{split}
			&\widehat{\CFDA}(\varphi')\boxtimes\widehat{\CFD}(H_0')\\&\cong(\widehat{\CFDA}(\varphi)\boxtimes\widehat{\CFD}(H_0))\otimes_\F(\widehat{\CFDA}(\mathbb{T}^2\times[0,1])\boxtimes\widehat{\CFD}(D^2\times S^1)).
		\end{split}
	\end{align}
	of $\mathcal{A}_2$-modules. Since $\mathbb{T}^2\times[0,1]\cup D^2\times S^1\cong D^2\times S^1$, by \cite[Lemma 4.2]{Hendricks2019}, there is an unique homogeneous homotopy equivalence
	\begin{align}
		h_1:\widehat{\CFD}(D^2\times S^1)\to\widehat{\CFDA}(\mathbb{T}^2\times[0,1])\boxtimes\widehat{\CFD}(D^2\times S^1).
	\end{align}
	\begin{figure}
		\begin{center}
			\includegraphics[scale=1]{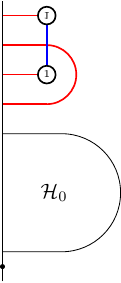}\hspace{0.75in}\includegraphics[scale=1]{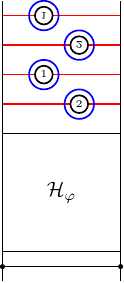}
		\end{center}
		\caption{Bordered Heegaard diagrams $\mathcal{H}_0'$ (left) and $\mathcal{H}_{\varphi'}$ (right) obtained by appending standard diagrams to $\mathcal{H}_0$ and $\mathcal{H}_\varphi$.}
		\label{Append}
	\end{figure}
	Now, the standard diagram for $D^2\times S^1$ with the 0-framing is
	\begin{center}
		\includegraphics[scale=1]{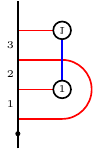}
	\end{center}
	which has one generator, $\bm{s}$, and supports a single disk
	\begin{center}
		\includegraphics[scale=1]{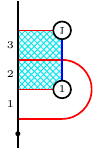}
	\end{center}
	with asymptotic condition $\rho_{23}\in\mathcal{A}(\mathbb{T}^2)$ so
	\begin{align}
		\widehat{\CFD}(D^2\times S^1)&=\begin{tikzcd}[ampersand replacement=\&]
			\bm{s}\arrow[loop right,looseness=8,out=35,in=-35,"\rho_{23}"]
		\end{tikzcd}
	\end{align}
	and, hence, $\widehat{\CF}(S^2\times S^1)\simeq\mathrm{End}^{\mathcal{A}(\mathbb{T}^2)}(\widehat{\CFD}(D^2\times S^1))=\F\langle\theta_1,\theta_2\rangle$, where $\theta_1(\bm{s})=\bm{s}$ and $\theta_2(\bm{s})=\rho_{23}\bm{s}$. One may easily check that $\partial\theta_1=2\theta_2=0$ and $\partial\theta_2=0$ so $\theta_1=\theta^{\mathrm{top}}$ and $\theta_2=\theta^{\mathrm{bot}}$. Under the above identifications, the 1-handle map
	\begin{align}
		\widehat{F}_{X_1}:\Mor^{\mathcal{A}_1}(-H_0,\varphi\boxtimes H_0)\to\Mor^{\mathcal{A}_2}(-H_0',\varphi'\boxtimes H_0')
	\end{align}
	is given by $f\mapsto f^{\mathrm{top}}$, where $f^{\mathrm{top}}=(\mathrm{id}\otimes h_1)\circ(f\otimes\theta^{\mathrm{top}})=f\otimes h_1$. The case of $\ell$ 1-handles is identical with the exception that one must instead append $k$ copies of the standard diagram for $D^2\times S^1$, in which case $\theta^{\mathrm{top}}=\theta_1^{\otimes \ell}$ and the codomain of $\widehat{F}_{X_1}$ is a space of morphisms of $\mathcal{A}(F(-\mathcal{Z})\#(\mathbb{T}^2)^{\#\ell})$-modules.
	
	For the 2-handle map $\widehat{F}_{X_2}:\widehat{\CF}(Y_0')\to\widehat{\CF}(Y_1')$, we needed some potentially different Heegaard splitting $Y_0'=H\cup_\psi H$ (we again assume that $H$ is 0-framed). However, by the Reidemeister--Singer theorem, after stabilizing sufficiently many times, we may arrange that $H_0'\cup_{\varphi'}H_0'$ and $H\cup_\psi H$ are isotopic Heegaard splittings so $H=H_0$ and $\psi=\xi^{-1}\circ\varphi'\circ\eta$, where $\eta,\xi:\partial H\to\partial H_0'$ are diffeomorphisms extending over $H=H_0$ (cf. \cite[Theorem 2.2]{Pitsch2008}). Then we may compute $\widehat{\CF}(Y_0')$ as 
	\begin{align}
		\begin{split}
			&\Mor^{\mathcal{A}_2}(-H,\psi\boxtimes H)\\&\cong\overline{\widehat{\CFD}(-H)}\boxtimes\mathcal{A}_2\boxtimes\widehat{\CFDA}(\psi)\boxtimes\widehat{\CFD}(H)\\&\simeq\overline{\widehat{\CFD}(-H)}\boxtimes\mathcal{A}_2\boxtimes\widehat{\CFDA}(\xi^{-1})\boxtimes\widehat{\CFDA}(\varphi')\boxtimes\widehat{\CFDA}(\eta)\boxtimes\widehat{\CFD}(H)\\&\simeq\overline{\widehat{\CFD}(-H)}\boxtimes\overline{\widehat{\CFDA}(-\xi^{-1})}\boxtimes\mathcal{A}_2\boxtimes\widehat{\CFDA}(\varphi')\boxtimes\widehat{\CFDA}(\eta)\boxtimes\widehat{\CFD}(H)\\&\cong\overline{\widehat{\CFDA}(-\xi^{-1})\boxtimes\widehat{\CFD}(-H)}\boxtimes\mathcal{A}_2\boxtimes\widehat{\CFDA}(\varphi')\boxtimes\widehat{\CFDA}(\eta)\boxtimes\widehat{\CFD}(H)\\&\simeq\overline{\widehat{\CFD}(-H_0')}\boxtimes\mathcal{A}_2\boxtimes\widehat{\CFDA}(\varphi')\boxtimes\widehat{\CFD}(H_0')\\&\cong\Mor^{\mathcal{A}_2}(-H_0',\varphi'\boxtimes H_0').
		\end{split}
	\end{align}
	Here, the homotopy equivalence in the third line is given to us by \cite[Lemma 4.5]{Hendricks2019}, which tells us that there is an unique homogeneous homotopy equivalence
	\begin{align}
		\widehat{\CFDA}(\psi)\simeq\widehat{\CFDA}(\xi^{-1})\boxtimes\widehat{\CFDA}(\varphi')\boxtimes\widehat{\CFDA}(\eta).
	\end{align}
	By \cite[Lemma 4.2]{Hendricks2019}, there are unique homogeneous homotopy equivalences $\widehat{\CFD}(H)\to\widehat{\CFD}(H_0')$ and $\widehat{\CFD}(-H)\to\widehat{\CFD}(-H_0')$ so this furnishes us with an algorithmically computable homotopy equivalence
	\begin{align}
		h_2:\Mor^{\mathcal{A}_2}(-H_0',\varphi'\boxtimes H_0')\to\Mor^{\mathcal{A}_2}(-H,\psi\boxtimes H)
	\end{align}
	of morphism complexes. Moreover, this map agrees up to homotopy with the homotopy equivalence associated to the map associated to a sequence of Heegaard moves (cf. \cite[proof of Theorem 5.1]{Hendricks2019}). The map
	\begin{align}
		\widehat{F}_{X_2}\circ\widehat{F}_{X_1}:\Mor^{\mathcal{A}_2}(-H_0,\varphi\boxtimes H_0)\to\Mor^{\mathcal{A}_2}(-H(\mathbb{L}),\psi\boxtimes H)
	\end{align}
	is then given by $\widehat{F}_{X_2}\circ\widehat{F}_{X_1}(f)=h_2(f^{\mathrm{top}})\circ\theta^{\mathrm{top}}$ by (\ref{square}).
	
	The case of 3-handles follows similarly to the case of 1-handles: if the cobordism $X_3:Y_1'\to Y_1$ consists of a single 3-handle addition, then $\widehat{F}_{X_3}:\widehat{\CF}(Y_1')\to\widehat{\CF}(Y_1)$ can be computed by decomposing $Y_1'$ as $Y_1\#(S^2\times S^1)$, in which case
	\begin{align}
		\widehat{F}_X(\bm{y}\otimes\theta)=\begin{cases}
			\bm{y}\,\,\,\,\textup{if $\theta=\Theta^{\mathrm{bot}}$}\\0\,\,\,\,\,\textup{else}.
		\end{cases}
	\end{align}
	In the morphism spaces setting, we leverage the fact that we have Heegaard splittings $Y_1'=H(\mathbb{L})\cup_\psi H=H_2'\cup_{\omega'}H_2'$, where $H_2'=H_2\natural(D^2\times S^1)$ and $\omega'=\omega\#\mathrm{id}_{\mathbb{T}^2}$ for some Heegaard splitting $Y_1=H_2\cup_\omega H_2$. As before, we may stabilize sufficiently many times so that $H(\mathbb{L})\cup_\psi H$ and $H_2'\cup_{\omega'}H_2'$ are isotopic Heegaard splittings and we obtain isomorphisms
	\begin{align}
		\widehat{\CFD}(H_2')\cong\widehat{\CFD}(H_2)\otimes_\F\widehat{\CFD}(D^2\times S^1)
	\end{align}
	and
	\begin{align}
		\begin{split}
			&\widehat{\CFDA}(\omega')\boxtimes\widehat{\CFD}(H_2')\\&\cong(\widehat{\CFDA}(\omega)\boxtimes \widehat{\CFD}(H_2))\otimes_\F(\widehat{\CFDA}(\mathbb{T}^2\times[0,1])\boxtimes\widehat{\CFD}(D^2\times S^1))
		\end{split}
	\end{align}
	of $\mathcal{A}(-\partial H_2\#\mathbb{T}^2)$-modules. There is then an unique homogeneous homotopy equivalence
	\begin{align}
		\begin{split}
			h_3:&\Mor^{\mathcal{A}(-\partial H_2\#\mathbb{T}^2)}(-H(\mathbb{L}),\psi\boxtimes H)\\&\to\Mor^{\mathcal{A}(-\partial H_2\#\mathbb{T}^2)}(-H_2\otimes_\F(D^2\times S^1),(\omega\boxtimes H_2)\otimes_\F(D^2\times S^1))
		\end{split}
	\end{align}
	induced by $h_1^{-1}$, which factors through $\Mor^{\mathcal{A}(-\partial H_2\#\mathbb{T}^2)}(-H_2',\omega'\boxtimes H_2')$
	so that the 3-handle map
	\begin{align}
		\widehat{F}_{X_3}:\Mor^{\mathcal{A}_2}(-H(\mathbb{L}),\psi\boxtimes H)\to\Mor^{\mathcal{A}_3}(-H_2,\omega\boxtimes H_2),
	\end{align}
	where $\mathcal{A}_3=\mathcal{A}(-\partial H_2)$, is then given by $((\mathrm{id}\otimes\overline{\theta}^{\mathrm{bot}})\circ h_3)(f)$, where $\overline{\theta}^{\mathrm{bot}}$ is the $\mathcal{I}$-linear dual of $\theta^{\mathrm{bot}}$. In summary, if $X=X_3\circ X_2\circ X_1$, we may compute the map $\widehat{F}_X$ at the chain level via $\widehat{F}_X(f)=((\mathrm{id}\otimes\overline{\theta}^{\mathrm{bot}})\circ h_3)(h_2(f^{\mathrm{top}})\circ\theta^{\mathrm{top}})$.
	
	Since each of the 1-, 2-, and 3-handle maps and the homotopy equivalences of morphism complexes at each step are algorithmically computable, Theorem \ref{MainTheorem} and \cite{Lipshitz2014} furnish us with an algorithm for computing $\widehat{F}_X$, whose steps we outline below:
	\begin{enumerate}
		\item Fix a Heegaard splitting $Y_0=H_0\cup_\varphi H_0$ which has been stabilized sufficiently many times so that all of the pairs of Heegaard splittings in each step described above become isotopic, then pick a factorization of the gluing map $\varphi$ into arcslides.
		\item Compute a basis $\{f_1,\dots,f_n\}$ for $H_*\Mor^{\mathcal{A}_1}(-H_0,\varphi\boxtimes H_0)$ consisting of explicit cycles in $\Mor^{\mathcal{A}_1}(-H_0,\varphi\boxtimes H_0)$.
		\item For each $f_i$, compute the map $f_i^{\mathrm{top}}\in\Mor^{\mathcal{A}_2}(-H_0',\varphi'\boxtimes H_0')$.
		\item Fix a (sufficiently stabilized) Heegaard splitting $Y_0'=H\cup_\psi H$ induced by a bouquet for a framed link $\mathbb{L}\subset Y_0'$ such that $Y_0'(\mathbb{L})=Y_1'$ and compute $\widehat{\CFD}(H)$ and a basis for $H_*\Mor^{\mathcal{A}_2}(H_0',H)$ in order to find the unique homogeneous homotopy equivalences which induce the homotopy equivalence $h_2$, and compute the latter.
		\item Compute $\widehat{\CFD}(-H(\mathbb{L}))$ and a basis for $H_*\Mor^{\mathcal{A}_2}(-H(\mathbb{L}),H)$ consisting of explicit cycles, identify $\theta^{\mathrm{top}}\in\Mor^{\mathcal{A}_2}(-H(\mathbb{L}),H)$ using this basis, and compute $h_2(f^{\mathrm{top}})\circ\theta^{\mathrm{top}}$.
		\item Compute $\widehat{\CFDA}(\psi)\boxtimes\widehat{\CFD}(H)$, a basis for $\Mor^{\mathcal{A}_2}(-H(\mathbb{L}),\psi\boxtimes H)$, and the homotopy equivalence $h_3$.
		\item Compute $\widehat{F}_X(f_i)=((\id\otimes\overline{\theta}^{\mathrm{bot}})\circ h_3)(h_2(f_i^{\mathrm{top}})\circ\theta^{\mathrm{top}})$ for $i=1,\dots,n$.
	\end{enumerate}
	\bibliographystyle{alpha}
	\bibliography{bib}
\end{document}